\documentclass[10pt,a4paper,twoside,reqno]{amsart}

\usepackage{amssymb}
\usepackage{mathtools}
\usepackage[shortlabels]{enumitem}
\usepackage[utf8]{inputenc}
\usepackage{color}
\usepackage{mathabx}
\usepackage{tikz}
\usetikzlibrary{positioning,decorations.pathreplacing,decorations.markings,arrows.meta,calc,fit,quotes,cd,math,arrows,backgrounds,shapes.geometric}
\usepackage[all,2cell]{xy}\UseAllTwocells\SilentMatrices
\usepackage[a4paper,top=3cm,bottom=3cm,inner=2.5cm,outer=2.5cm]{geometry}
\definecolor{darkgreen}{rgb}{0,0.45,0}
\usepackage[colorlinks,citecolor=darkgreen,urlcolor=gray,final,hyperindex,linktoc=page,hyperfootnotes=true]{hyperref}
\usepackage[capitalize]{cleveref}
\usepackage{mathrsfs}
\usepackage{comment}

\crefname{equation}{}{}
\crefname{thm}{Theorem}{Theorems}
\crefname{defi}{Definition}{Definitions}
\crefname{rmk}{Remark}{Remarks}
\crefname{prop}{Proposition}{Propositions}
\crefname{ex}{Example}{Examples}
\crefformat{footnote}{#2\footnotemark[#1]#3}

\theoremstyle{plain}
\newtheorem{thm}{Theorem}[section]
\newtheorem{cor}[thm]{Corollary}
\newtheorem{lem}[thm]{Lemma}
\newtheorem{prop}[thm]{Proposition}

\theoremstyle{remark}
\newtheorem{rmk}[thm]{Remark}
\newtheorem{ex}[thm]{Example}

\theoremstyle{definition}
\newtheorem{defi}[thm]{Definition}

\tikzstyle{start}=[to path={(\tikztostart.#1) -- (\tikztotarget)}]
\tikzstyle{end}=[to path={(\tikztostart) -- (\tikztotarget.#1)}]

\newcommand{\ca}{\mathcal}

\newcommand{\nc}{\mathsf}

\newcommand{\Set}{\nc{Set}}
\newcommand{\Pos}{\nc{Pos}}

\tikzset{tick/.style={postaction={decorate,decoration={markings,mark=at position 0.5 with {\draw[-] (0,.5ex) -- (0,-.5ex);}}}}}
\tikzset{bul/.style={postaction={decoration={markings,mark=at position 0.5 with {\node{$\sbul$};}},decorate}}}

\tikzset{tick/.style={postaction={decorate,decoration={markings,mark=at position 0.5 with {\draw[-] (0,.4ex) -- (0,-.4ex);}}}}}

\newcommand{\vleq}{\mathrel{\rotatebox{270}{$\leq$}}}

\begin{document}
\title{Poset-enriched categories and free exact completions}
	
\author{Vasileios Aravantinos-Sotiropoulos}

\address{School of Applied Mathematical and Physical Sciences, National Technical University of Athens, Greece}
\email{v\_aravantinos@mail.ntua.gr}

\begin{abstract}
	We give an elementary construction of the exact completion of a weakly lex category for categories enriched in the cartesian closed category $\mathsf{Pos}$ of partially ordered sets. Paralleling the ordinary case, we characterize categories which arise as such completions in terms of projective objects. We then apply the results to categories of Eilenberg-Moore algebras for enriched monads on $\mathsf{Pos}$. In particular, we show that every variety of ordered algebras is the exact completion of a subcategory on certain free algebras, thereby answering a question of A. Kurz and J. Velebil.
\end{abstract}

\maketitle

\section{Introduction}

The notion of regularity and the stronger Barr-exactness are by now ubiquitous in the literature on Category Theory and have found many applications in different areas such as Algebra and Logic. First considered in \cite{barr1971exact}, exactness can be thought of as the non-additive part of the axioms defining an abelian category. This manifests in Tierney's theorem, which states that a category is abelian if and only if it is both exact and additive.  However, there are at least a few perspectives from which these notions can be approached.

From one perspective, a finitely complete category is regular when it has quotients for kernel pairs and these behave well, in the sense that they are stable under pullback. From another, regularity is the existence of stable factorizations of morphisms into a regular epimorphism followed by a monomorphism, while a third states that regular categories are those which afford a useful calculus of internal relations. Exact categories are then those which are regular and in which, in addition, every equivalence relation is \emph{effective}, i.e. a kernel pair of some morphism.

There are many interesting contexts which can be described as a regular or exact category satisfying some additional axioms. We mentioned abelian categories above, but exactness holds in elementary toposes as well, which save for this are of a quite different nature than the former. It is also present in the varieties of Universal Algebra (in fact, in all categories \emph{monadic} over $\nc{Set}$), while regularity holds more generally in quasi-varieties. Actually, the distinction between the two algebraic notions corresponds precisely to the difference between the two categorical ones.

Given the significance of regular and exact categories, it is natural to ask whether a category lacking these properties can be ``completed'' to one which does not. The first such consideration is due to A. Carboni \& R. Celia-Magno \cite{CarboniCelia-Magno}, who gave a construction of the \emph{free} exact category $\ca{C}_{\rm{ex}}$ on a left exact category $\ca{C}$. Here ``free'' has its usual meaning, i.e. that it provides a left adjoint to a certain forgetful functor. In particular, this means that there is an exact category $\ca{C}_{\rm{ex}}$ together with a lex embedding $\ca{C}\hookrightarrow\ca{C}_{\rm{ex}}$ such that there is an equivalence
\begin{displaymath}
	\nc{Lex}(\ca{C},\ca{E})\simeq\nc{Reg}(\ca{C}_{\rm{ex}},\ca{E})
\end{displaymath}
between lex functors $\ca{C}\to\ca{E}$ and regular functors $\ca{C}_{\rm{ex}}\to\ca{E}$, for any exact category $\ca{E}$. 

It was later realized by Carboni that the construction of the exact completion $\ca{C}_{\rm{ex}}$ can still be carried out even if $\ca{C}$ is merely \emph{weakly} lex, which is to say that it only possesses \emph{weak} finite limits. Then Carboni \& Vitale \cite{Carboni-Vitale} produced constructions for both the regular and exact completion of a weakly lex category together with the associated universal properties. What is noteworthy here is that these regular and exact completions of a weakly lex $\ca{C}$ are no longer part of left adjoints to the obvious forgetful functors. In particular, it is not weakly lex functors (=those preserving weak finite limits) $\ca{C}\to\ca{E}$ into exact categories $\ca{E}$ that are involved in the universal property of $\ca{C}_{\rm{ex}}$. Rather, it is those functors which E. Vitale discovered and termed \emph{left covering}\cite{VitaleThesis}, a property which has since been identified as a form of flatness (see \cite{NotionsOfFlatness,tendas2024flatness}).

\vspace{0.3cm}

On the other hand, in a more recent paper \cite{Kurz-Velebil}, A. Kurz and J. Velebil considered for the first time notions of regularity and exactness specifically suited to the context of categories enriched over the cartesian monoidal category $\nc{Pos}$ of partially ordered sets and monotone functions. The main result therein is a pair of characterization theorems for quasivarieties and varieties respectively of ordered algebras in the sense of S. Bloom \& J. Wright \cite{BloomWright}. One aspect of these results which is remarkable is how they precisely mirror known results concerning ordinary varieties of algebras. More precisely, varieties of ordered algebras can be characterized as exact $\nc{Pos}$-categories which have a suitably nice generator, while for quasivarieties the corresponding result simply replaces exactness with regularity. In other words, the characterization of ordered (quasi-)varieties reads identical to the one of ordinary (quasi-)varieties, as long as regularity and exactness are interpreted in the appropriate enriched sense.

The work of Kurz \& Velebil has been the starting point for extensive research into the realm of $\nc{Pos}$-enriched categories in recent years, as for example in \cite{Adamek_Dostal_Velebil_2022,Adamek2023varieties,Adamek2024TAC,OrdMaltsev}. In particular, in \cite{MyPosExReg} we used a suitable calculus of relations to construct the exact completion of a regular $\nc{Pos}$-enriched category. The latter gives an enriched analogue of the ordinary ex/reg completion, which was initially suggested by Lawvere \cite{PerugiaNotes} and first described in detail by R. Succi Cruciani \cite{SucciCruciani}. However, (regular and) exact completions of weakly lex categories have not yet been considered in this context. In fact, one of the three open questions in \cite{Kurz-Velebil} concerns the existence of such completions and their connection to (quasi-)varieties of ordered algebras.

The main purpose of this paper is to attempt to provide an answer to this question. By adapting the construction of Carboni-Vitale from the ordinary setting, we give a construction of the exact completion $\ca{C}_{\rm{ex}}$ of any weakly lex $\nc{Pos}$-category $\ca{C}$. Here the notion of ``weakly lex'' category can be seen as an instance of a general framework for weak notions described by Lack \& Ros\'icky \cite{LackRosickyWeakness}. The universal property of $\ca{C}_{\rm{ex}}$ again involves left covering functors, where the latter is interpreted in the obvious enriched sense. In addition, we characterize free exact categories. What is once more noteworthy is how the results precisely mimic those for ordinary categories. In particular, every exact $\nc{Pos}$-category with enough projectives is the exact completion of any of its projective covers. Finally, we observe that in every variety of ordered algebras there is a projective cover given by those algebras which are free on a discrete poset. Thus, every variety is the exact completion of the full subcategory on such free algebras. In fact, the same holds more generally for categories $\nc{Pos}^{\mathbb{T}}$ for a sufficiently well behaved monad $\mathbb{T}$.

\vspace{0.3cm}

\emph{Outline of the paper}

In Section 2 we quickly record some preliminaries on $\nc{Pos}$-categories which will be required in the rest of the paper. For the most part, we do not go into much detail, so the reader should consult the indicated references if need be.

In Section 3, starting with a weakly lex $\nc{Pos}$-category $\ca{C}$, we describe a construction of the exact completion $\ca{C}_{\rm{ex}}$ and prove that it indeed produces an exact category. We also describe the canonical functor $\ca{C}\to\ca{C}_{\rm{ex}}$ and make a couple of observations on the manner in which $\ca{C}$ sits inside $\ca{C}_{\rm{ex}}$.

The universal property of the completion is established in Section 4, which also contains the characterization of free exact categories via projectives. The latter is applied in Section 5 to show that varieties of ordered algebras are exact completions and more generally to deduce a type of monadicity theorem over $\nc{Pos}$.

\vspace{0.3cm}

\emph{Acknowledgements}

The author would like to thank the members of the Category Theory group of Universit\'e Catholique de Louvain for listening to a talk on the subject of this paper and providing useful feedback.

This paper draws from and owes its existence to the work of Enrico Vitale and in particular to his wonderfully written thesis \cite{VitaleThesis}. Indeed, it was the detail and clarity in the latter that convinced the author that it should be possible to adapt the ideas to the poset-enriched context and also guided his way in attempting to carry this out.

\section{Preliminaries on $\Pos$-enriched Categories}

A \emph{poset-enriched category} $\ca{C}$ is a category enriched in the cartesian monoidal category $\ca{V}=\Pos$ of partially ordered sets and monotone functions. In more explicit terms this means that for any two objects $X,Y\in\ca{C}$ the set $\ca{C}(X,Y)$ of morphisms $X\to Y$ is equipped with a partial order $\leq$ and composition of morphisms is monotone in both variables. Similarly, $\Pos$-enriched functor $F\colon\ca{C}\to\ca{D}$ is a functor between $\Pos$-enriched categories which is \emph{locally monotone}, i.e. such that for all objects $X,Y\in\ca{C}$ the function $\ca{C}(X,Y)\to\ca{D}(FX,FY)$ is order-preserving.

In this paper, the terms \emph{category} and \emph{functor} will always be used to indicate the $\Pos$-enriched notions just recalled. If we want to consider the non-enriched (or $\Set$-enriched) notions, we will explicitly use the prefix \emph{ordinary}.

\subsection{Finite Weighted Limits}

Throughout the paper we shall use the term ``limit'' to mean \emph{weighted} limit. Hence, over our base $\ca{V}=\nc{Pos}$, the \emph{limit} of the diagram $D\colon\ca{D}\to\ca{C}$ weighted by $W\colon\ca{D}\to\nc{Pos}$ is an object $\{W,D\}$ with an isomorphism
\begin{displaymath}
	\ca{C}(X,\{W,D\})\cong[\ca{D},\nc{Pos}](W,\ca{C}(X,D-))
\end{displaymath}
which is natural in $X$.

If the weight $W\colon\ca{D}\to\nc{Pos}$ is taken to be the constant functor $\Delta1\colon\ca{D}\to\nc{Pos}$ on the one element poset, then we speak of a \emph{conical} limit. An example of such would be the equalizer $e\colon E\to X$ of a parallel pair $\begin{tikzcd}
	X\ar[r,shift left=0.75ex,"f"]\ar[r,shift right=0.75ex,"g"'] & Y
\end{tikzcd}$. Unraveling the definition, this means that $e$ satisfies the familiar universal property from ordinary category theory, but that in addition it is order-monic (as opposed to simply monic) in the following sense.

\begin{defi}
	A morphism $m\colon X\to Y$ in a category $\ca{C}$ is called an \emph{$\nc{ff}$-morphism} or an \emph{order-monomorphism} if for every $Z\in\ca{C}$ the monotone map $\ca{C}(Z,m)\colon\ca{C}(Z,X)\to\ca{C}(Z,Y)$ in $\nc{Pos}$ also reflects the order.
\end{defi}

A similar remark applies to any type of conical limit. Namely, the universal property will be the same as in the ordinary case, but augmented by the requirement that the legs of the limit cone are jointly order-monomorphic.

In this paper, we shall particularly like to consider \emph{finite} limits. Indeed, there is a standard notion of finite limit in enriched category theory which is due to Kelly \cite{FiniteLimitsEnriched}. Over the base category $\nc{Pos}$, a limit $\{W,D\}$ is called finite if the weight $W\colon\ca{D}\to\ca{C}$ satisfies the following:
\begin{enumerate}
	\item $\ca{D}$ has finitely many (isomorphism classes of) objects.
	\item $\ca{D}(d,d')$ is a finite poset for all $d,d'\in\ca{D}$.
	\item $Wd$ is a finite poset for all $d\in\ca{D}$.
\end{enumerate}
We now quickly recall some examples of (co)limits that will appear in the remainder of the paper. All except the last one are finite.

\begin{itemize}
\item The \emph{comma object} (also called \emph{lax pullback}) of an ordered pair of morphisms $(f\colon X\to Z,g\colon Y\to Z)$ is a square
\begin{displaymath}
	\begin{tikzcd}[sep=0.35in]		f/g\ar[r,"c_{1}"]\ar[d,"c_{0}"']\ar[dr,phantom,"\leq"] & Y\ar[d,"g"] \\
		X\ar[r,"f"'] & Z
	\end{tikzcd}	
\end{displaymath}
such that $fc_{0}\leq gc_{1}$ and which is universal with this property, in the following sense:
\begin{enumerate}
	\item Given $u_{0}:W\to X$ and $u_{1}\colon W\to Y$ in $\ca{C}$ such that $fu_{0}\leq gu_{1}$, there exists a $u\colon W\to f/g\in\ca{C}$ such that $c_{0}u=u_{0}$ and $c_{1}u=u_{1}$.
	\item If $h,h'\colon A\to f/g$ are such that $c_{0}h\leq c_{0}h'$ and $c_{1}h\leq c_{1}h'$, then $h\leq h'$.
\end{enumerate}
\item In particular, the \emph{kernel congruence} (or \emph{lax kernel}) of a morphism $f\colon X\to Y$ is the comma object $f/f$ of $f$ with itself.
\item The \emph{inserter} of an ordered pair $(f,g)$ of parallel morphisms $\begin{tikzcd}X\ar[r,shift left=1ex,"f"]\ar[r,shift right=1ex, "g"']& Y\end{tikzcd}$ is a morphism $e\colon E\to X\in\ca{C}$ which is universal such that $fe\leq ge$.
\item The \emph{coinserter} of an ordered pair $(f,g)$ of parallel morphisms $\begin{tikzcd}X\ar[r,shift left=1ex,"f"]\ar[r,shift right=1ex, "g"']& Y\end{tikzcd}$ is a morphism $q\colon Y\to Q\in\ca{C}$ which is universal such that $qf\leq qg$.
\item The \emph{copower} (also called \emph{tensor}) of an object $C\in\ca{C}$ to a poset $P\in\nc{Pos}$ is an object $P\bullet C$ defined by the natural isomorphism
\begin{displaymath}
	\ca{C}(P\bullet C,X)\cong\nc{Pos}(P,\ca{C}(C,X))
\end{displaymath}
\end{itemize}

\subsection{Regular and Exact Categories}

The notions of regularity and exactness pertinent to the $\nc{Pos}$-enriched context can be considered as simpler versions of corresponding 2-categorical notions that go back to the work of Street \cite{2dimSheafTheory}. However, they were first considered on their own merits by Kurz \& Velebil in their pioneering paper \cite{Kurz-Velebil}. In this section we quickly recall these notions. The reader can also consult \cite[Chapter 4]{MyThesis} for details.

The appropriate notion of monomorphism has already been recalled above. For the epimorphic part we take the class of morphisms which are left orthogonal to the order-monomorphisms, in the following sense.

\begin{defi}
We say that $e\colon A\to B$ is an \emph{$\nc{so}$-morphism} if the square
	\begin{center}
		\begin{tikzcd}
			\ca{C}(B,X)\ar[r,"-\circ e"]\ar[d,"m\circ -"'] & \ca{C}(A,X)\ar[d,"m\circ -"] \\
			\ca{C}(B,Y)\ar[r,"-\circ e"'] & \ca{C}(A,Y)
		\end{tikzcd}
	\end{center}
	is a pullback in $\nc{Pos}$ for every order-mono $m\colon X\to Y$.
\end{defi}

We can then define regularity as follows.

\begin{defi}\label{defi: regular cat}
	A category $\ca{C}$ is called \emph{regular} if it has all finite limits, every morphism factors as an $\nc{so}$-morphism followed by an order-monomorphism and $\nc{so}$-morphisms are stable under pullback.
\end{defi}

We should note that the original definition in \cite{Kurz-Velebil} contains the additional requirement that every $\nc{so}$-morphism be an \emph{effective} epimorphism, i.e. a coinserter of some pair, so then necessarily of its kernel congruence. However, one can then prove (see \cite[Proposition 2.16]{MyPosExReg}) that this actually is implied by the remaining conditions of \ref{defi: regular cat}. In particular, it then follows that one can also phrase regularity in terms of stable (effective epi, order-mono) factorizations or via the existence and stability of coinserters of kernel congruences. Throughout the paper we will use the terms ``$\nc{so}$-morphism'' and ``effective epimorphism'' interchangeably in the regular context.

Examples of regular categories are all quasivarieties of ordered algebras in the sense of Bloom \& Wright \cite{BloomWright,Kurz-Velebil}, as well as categories of enriched presheaves $[\ca{C}^{\rm{op}},\nc{Pos}]$ on any small category $\ca{C}$. In addition, certain categories of ordered topological spaces, such as the category $\nc{Pries}$ of \emph{Priestley spaces} or the larger category $\nc{Nach}$ of Nachbin's \emph{compact ordered spaces}, are also regular in this sense.

To talk about exactness, we first need to recall the corresponding notion of congruence in the $\nc{Pos}$-enriched setting. These will be relations which are preorders and are in some sense compatible with the existing order furnished by the enrichment. By a \emph{relation} on an object $X$ we as usual mean a subobject $R\subseteq X\times X$, i.e. an equivalence class of an order-monomorphism $R\rightarrowtail X\times X$.

\begin{defi}
	A \emph{congruence} $E$ on an object $X$ in a regular category $\ca{C}$ is a relation on $X$ which is:
	\begin{enumerate}
		\item reflexive
		\item transitive
		\item an \emph{order-ideal}, i.e. whenever $x,x',y,y'\colon A\to X$ are such that $(x,y)\in_{A}E$, $x'\leq x$ and $y\leq y'$, then also $(x',y')\in_{A}E$.
	\end{enumerate}
\end{defi}

We note here for future use that properties (1) and (3) can be combined into one. Namely, a relation $E$ on $X$ is a congruence if and only if it is transitive and \emph{order-reflexive} (see \cite{MyPosExReg}). The latter means that $E$ contains not merely the diagonal $\Delta_X$ but the full comma object $I_X\coloneqq 1_X/1_X$. In other words, it is to say that any pair of morphisms $x,x'\colon A\to X$ with $x\leq x'$ must factor through $E$.

\begin{rmk}
	In \cite{MyPosExReg}, the term \emph{weakening-closed} was used for relations satisfying property (3) above. We prefer here to switch to \emph{order-ideal}, since the latter is more established in the literature.
\end{rmk}

A congruence is \emph{effective} if it is the kernel congruence of some morphism. The definition of exactness can then be stated as follows.

\begin{defi}\cite{Kurz-Velebil}
	A category $\ca{E}$ is called \emph{exact} if it is regular and every congruence in $\ca{E}$ is effective.
\end{defi}

Of the examples mentioned earlier, exact are all varieties of ordered algebras, all presheaf categories $[\ca{C}^{\rm{op}},\nc{Pos}]$ and the category of Nachbin spaces $\nc{Nach}$.

\subsection{Projective Covers and Weak Weighted Limits}

For ordinary categories, projective objects are intimately connected to both the regular and exact completion of a category with (weak) finite limits. The exact same connection is present in our enriched setting. Of course here projectivity should be defined in terms of the appropriate notion of epimorphism.

\begin{defi}
	An object $P$ in a regular category $\ca{C}$ will be called $(\nc{so}-)$\emph{projective} if $\ca{C}(P,-)\colon\ca{C}\to\nc{Pos}$ preserves $\nc{so}$-morphisms.
\end{defi}

\begin{defi}
	A full subcategory $\ca{P}$ of a regular category $\ca{C}$ is called a \emph{projective cover} of $\ca{C}$ if:
	\begin{enumerate}
		\item Every $P\in\ca{P}$ is projective in $\ca{C}$.
		\item For every $C\in\ca{C}$ there exist an object $P\in\ca{P}$ and an $\nc{so}$-morphism $P\twoheadrightarrow C$.
	\end{enumerate} 
\end{defi}

A category which has a projective cover is said to have \emph{enough projectives}. We will show later on that every variety of ordered algebras and more generally every category of algebras $\nc{Pos}^{\mathbb{T}}$ for a sufficiently nice monad $\mathbb{T}$ enjoy this property.

\begin{rmk}
Projectives in $\nc{Pos}$-enriched categories are often closely related to discrete objects. In particular, in any category which has enough discrete objects, every projective object must itself be discrete, being a subobject of a discrete one. This is the case, for example, in the variety $\nc{OrdMon}$ of ordered monoids and in $\nc{Nach}$.
\end{rmk}

Next, we need to make precise what our notion of weak weighted limit will be. In layman's terms, this can be done almost immediately: a weak limit of some given type will satisfy only the ``existence part'' of the universal property which defines the corresponding strong limit.

So for example, a \emph{weak inserter} of an ordered pair $(f,g)$ of parallel morphisms $\begin{tikzcd}X\ar[r,shift left=1ex,"f"]\ar[r,shift right=1ex, "g"']& Y\end{tikzcd}$ is a morphism $e\colon E\to X\in\ca{C}$ such that:

\begin{itemize}
	\item $fe\leq ge$.
	\item Given any $h\colon Z\to X$ such that $fh\leq gh$, there exists a $u\colon Z\to E$ with $eu=h$.
\end{itemize} 
So we've essentially dropped the requirement that $e$ be an order-monomorphism. The same observation applies to any other type of limit. However, to alleviate the feeling that this is somewhat of an ad hoc choice and with an eye towards other bases of enrichment over which one might want to consider similar constructions of exact completions, let us say a few more words.

There is a very general framework for weak notions in enriched category theory, so in particular weak limits, which is due to Lack \& Ros\'icky \cite{LackRosickyWeakness}. The data one needs consist of a monoidal category $\ca{V}$ and a class of morphisms $\ca{E}$ therein, which are required to satisfy some conditions. Then given a $\ca{V}$-functor $D\colon\ca{D}\to\ca{C}$ and a weight $W\colon\ca{D}\to\ca{V}$, a \emph{weak $W$-weighted limit} consists of:
\begin{itemize}
	\item An object $\{W,D\}_{w}$
	\item A natural transformation $W\to\ca{C}(\{W,D\}_{w},D-)$
\end{itemize}
such that for all $X\in\ca{C}$ the induced morphism $\ca{C}(X,\{W,D\}_w)\to[\ca{D},\ca{V}](W,\ca{C}(X,D-))$ lies in the class $\ca{E}$.

Applying this general definition to the case where $\ca{V}=\nc{Pos}$ and $\ca{E}$ is the class of surjections, we have the following definition.

\begin{defi}\label{defi: weak limits}
	Given a functor $D\colon\ca{D}\to\ca{C}$ and a weight $W\colon\ca{D}\to\nc{Pos}$, a \emph{weak $W$-weighted limit} consists of:
	\begin{itemize}
		\item An object $\{W,D\}_{w}$
		\item A natural transformation $W\to\ca{C}(\{W,D\}_{w},D-)$
	\end{itemize}
	such that for all $X\in\ca{C}$ the induced morphism $\ca{C}(X,\{W,D\}_w)\to[\ca{D},\nc{Pos}](W,\ca{C}(X,D-))$ is a surjection.
\end{defi}

It is then easy to see that, for example, in the case of the weight defining an inserter, $\{W,D\}_w$ will have precisely the property we mentioned above. In this paper we will be interested in categories which have weak finite limits, so we accordingly introduce the following.

\begin{defi}\label{defi: weakly lex}
	A category $\ca{C}$ is said to be \emph{weakly lex} if the weak limit $\{W,D\}_w$ exists for every finite weight $W\colon\ca{D}\to\nc{Pos}$.
\end{defi}

In order to have a better handle on the existence of weak finite limits, it will be useful to know how to construct all of them from specific types of such. In particular, we will need the following result.

\begin{prop}\label{weakly lex from weak prod and weak ins}
	If a category $\ca{C}$ has weak finite products and weak inserters, then it is weakly lex.
\end{prop} 
\begin{proof}
We will only describe the construction of an arbitrary weak finite limit and leave as an exercise for the reader the verification that this works. 
	
We observe first that weak equalizers can be constructed by taking two successive inserters. Then we can construct $n$-ary pullbacks via these and finite products.
	
Now consider an arbitrary finite weight $W\colon\ca{I}\to\nc{Pos}$ and a diagram $D\colon\ca{I}\to\ca{C}$. For every object $i\in\ca{I}$ and every $x\in Wi$, we set $D_{(i,x)}\coloneqq Di$ and form a weak product $P=\prod\limits_{i\in\ca{I},x\in Wi}D_{(i,x)}$. Then for every pair of arrows $\begin{tikzcd}
	i\ar[r,"\alpha"] & j & i'\ar[l,"\beta"']
\end{tikzcd}$ 
with common codomain and every pair $(x,x')\in Wi\times Wi'$ with $W\alpha(x)\leq W\beta(x')$ we take a weak inserter 
\begin{tikzcd}
	E^{(\alpha,\beta)}_{(x,x')}\ar[r,"e^{(\alpha,\beta)}_{(x,x')}"] & P\ar[r,shift left=1ex,"D\alpha\circ\pi_{(i,x)}"]\ar[r,shift right=1ex,"\vleq","D\beta\circ\pi_{(i',x')}"'] & Dj
\end{tikzcd}
We can assume here that for all $i\in\ca{I}$ and $x\in Wi$ we have chosen $e^{(1_i,1_i)}_{(x,x)}=1_P$. Finally, we form a weak pullback $l^{(\alpha,\beta)}_{(x,x')}\colon L\to E^{(\alpha,\beta)}_{(x,x')}$ of the finite family of arrows $e^{(\alpha,\beta)}_{(x,x')}\colon E^{(\alpha,\beta)}_{(x,x')}\to P$. Then $L$ with $\phi_i\colon Wi\to\ca{C}(L,Di)$ defined by $\phi_i(x)\coloneqq \pi_{(i,x)}\circ l^{(1_i,1_i)}_{(x,x)}$ gives the desired weak weighted limit.
\end{proof}

It is also not hard to check that, much like in the ordinary case, limits in a category imply the existence of the corresponding weak limits in any projective cover.

\begin{prop}\label{prop: weak limits in proj cover}
	Let $\ca{C}$ be a category and $\ca{P}$ a projective cover of $\ca{C}$. If $W\colon\ca{D}\to\nc{Pos}$ is any weight for which $\ca{C}$ has $W$-weighted limits, then $\ca{P}$ has weak $W$-weighted limits. In particular, if $\ca{C}$ is lex, then $\ca{P}$ is weakly lex.
\end{prop}

\section{The Exact Completion}

In this section we give the construction of the category $\ca{C}_{\rm{ex}}$ and the verification that it is indeed exact. To make the proof easier to digest, we have split it up into a sequence of separate steps verifying each successive layer of properties.

Our intent is to closely follow the construction of the ordinary exact completion of a weakly lex category, as presented in \cite{Carboni-Vitale} and especially its detailed exposition in \cite{VitaleThesis}. To begin with, the objects of $\ca{C}_{\rm{ex}}$ are weak versions of equivalence relations, called \emph{pseudo-}equivalence relations, which are the appropriate notion of congruence in the ordinary setting. By ``weak'' here we mean the absence of the jointly monic property of the span representing a relation. Thus, it seems natural that the corresponding objects to consider here should be the following.

\begin{defi}
A  \emph{pseudocongruence} in $\ca{C}$ is a parallel pair of morphisms 
$\begin{tikzcd}
	R \ar[r,"r_0",shift left=1ex]\ar[r,"r_1"',shift right=1ex] & X
	\end{tikzcd}$
\end{defi}
which is:
\begin{enumerate}
\item \emph{order-reflexive}: every pair of morphisms $a_{0},a_{1}\colon A\to X$ with $a_0\leq a_1$ factors through $(r_0,r_1)$.
\item \emph{transitive}: For a weak pullback square 
\begin{center}
	\begin{tikzcd}
		R*R\ar[r,"d_{0}"]\ar[d,"d_{1}"'] & R\ar[d,"r_{1}"] \\
		R\ar[r,"r_{0}"'] & X
	\end{tikzcd}
\end{center}
	there exists a $\tau\colon R*R\to R\in\ca{C}$ such that $r_{0}\tau=r_{0}d_{0}$ and $r_{1}\tau=r_{1}d_{1}$.
\end{enumerate}

\begin{rmk}
\begin{enumerate}
	\item Observe that transitivity does not depend on a choice of weak pullback, since any two such choices will factor through each other.
	\item Order-reflexivity of 
	$\begin{tikzcd}
		R \ar[r,"r_0",shift left=1ex]\ar[r,"r_1"',shift right=1ex] & X
		\end{tikzcd}$ is equivalent to the requirement that any (hence every) weak comma square
	\begin{center}
		\begin{tikzcd}
			I_X\ar[r,"\iota_{1}"]\ar[d,"\iota_{0}"']\ar[dr,phantom,"\leq"] & X\ar[d,equal] \\
			X\ar[r,equal] & X
		\end{tikzcd}
	\end{center}
		factors through $(r_0,r_1)$.
\end{enumerate}
\end{rmk}
\vspace{0.2cm}

We now define a category $\ca{C}_{\rm{ex}}$ as follows:
\begin{itemize}
	\item \underline{Objects}: The pseudocongruences of $\ca{C}$. The object $\begin{tikzcd}
		R \ar[r,"r_0",shift left=0.75ex]\ar[r,"r_1"',shift right=0.75ex] & X
	\end{tikzcd}$ will be denoted by a pair $(X,R)$. This is to be thought of as the would-be quotient of $X$ by $R$.
	\item \underline{Morphisms}: Given objects $(X,R)$ and $(Y,S)$, to define morphisms $(X,R)\to(Y,S)$ we first consider morphisms $f\colon X\to Y\in\ca{C}$ for which there exists an $\bar{f}\colon R\to S\in\ca{C}$ such that $s_{0}\bar{f}=fr_{0}$ and $s_{1}\bar{f}=fr_{1}$.
	\begin{center}
		\begin{tikzcd}
			R \ar[r,"r_0",shift left=0.75ex]\ar[r,"r_1"',shift right=0.75ex]\ar[d,"\bar{f}"'] & X\ar[d,"f"] \\
			S \ar[r,"s_0",shift left=0.75ex]\ar[r,"s_1"',shift right=0.75ex] & Y
		\end{tikzcd}
	\end{center}
	Given our intuition on the objects, this is to say that we consider those morphisms which would be (well-defined and) order-preserving on the corresponding quotients.
	We then define a relation $f\preccurlyeq g$ on the collection of such morphisms as follows: $$f\preccurlyeq g \iff (\exists \Sigma\colon X\to S\in\ca{C})\quad s_{0}\Sigma=f, s_{1}\Sigma=g$$
	
	This is easily seen to be a preorder relation, reflexivity and transitivity following from the two corresponding internal properties of $S$. We denote by $\sim$ the induced equivalence relation and then define a morphism $[f]\colon (X,R)\to(Y,S)\in\ca{C}_{\rm{ex}}$ to be an equivalence class with respect to $\sim$.
	
	It is easy to check that $\preccurlyeq$ is compatible with composition of morphisms in $\ca{C}$ so that $\ca{C}_{\rm{ex}}$ canonically becomes a $\nc{Pos}$-category. 
\end{itemize}

We now begin the painstaking process of verifying that $\ca{C}_{\rm{ex}}$ is exact. In all proofs that follow, whenever we take a (weak) limit of some finite diagram in $\ca{C}$, we will denote the corresponding cone by broken arrows, while for the diagram itself we will use solid ones. Hopefully this alleviates some of the notational  complexity of the situation.

\begin{prop}\label{inserters in C_ex}
	$\ca{C}_{\rm{ex}}$ has inserters.
\end{prop}
\begin{proof}
	Consider an ordered pair of parallel morphisms $\begin{tikzcd}(X,R)\ar[r,"{[f]}",shift left=0.7ex]\ar[r,"{[g]}"',shift right=0.7ex] & (Y,S)
	\end{tikzcd}$ in $\ca{C}_{\rm{ex}}$. We then successively form the following two weak conical limits in $\ca{C}$:
	\begin{center}
		\begin{tikzcd}
			& X\ar[dl,"f"']\ar[dr,"g"] & \\
			Y & E\ar[u,dashed,"e"']\ar[d,dashed,"\phi"] & Y \\
			& S\ar[ul,"s_0"]\ar[ur,"s_1"'] & 
		\end{tikzcd}
		\qquad
		\begin{tikzcd}
			& \tilde{R}\ar[dl,dashed,"\tilde{r_0}"']\ar[dr,dashed,"\tilde{r_1}"]\ar[dd,dashed,"\bar{e}"'] & \\
			E\ar[dd,"e"'] & &  E\ar[dd,"e"] \\
			& R\ar[dl,"r_0"']\ar[dr,"r_1"] & \\
			X & & X
		\end{tikzcd}
	\end{center}
	
	We claim that $\begin{tikzcd}[cramped] \tilde{R}\ar[r,"\tilde{r_0}",shift left=0.75ex]\ar[r,"\tilde{r_1}"',shift right=0.75ex] & E
	\end{tikzcd}$ is a pseudocongruence and that then $[e]\colon(E,\tilde{R})\to(X,R)$ is a morphism in $\ca{C}_{\rm{ex}}$ which is the inserter of $([f],[g])$.
	
	First, consider any $a_{0}\leq a_{1}\colon A\to X$ in $\ca{C}$. Then $ea_0\leq ea_1$ and hence by order-reflexivity of $R$ we obtain a $u\colon A\to R\in\ca{C}$ such that $r_{0}u=ea_0$ and $r_{1}u=ea_1$. This means that $A$ together with the morphisms $a_0$, $a_1$ and $u$ form a cone over the diagram whose weak limit is $\tilde{R}$. Thus, we obtain a $u'\colon A\to\tilde{R}$ such that $\tilde{r_0}u'=a_0$, $\tilde{r_1}u'=a_1$ and $\bar{e}u'=u$.
	
	For transitivity, consider the following two weak pullbacks in $\ca{C}$
	\begin{center}
		\begin{tikzcd}
			R*R\ar[r,"d_{0}^{R}"]\ar[d,"d_{1}^{R}"'] & R\ar[d,"r_{1}"] \\
			R\ar[r,"r_{0}"'] & X
		\end{tikzcd}
		\quad
		\begin{tikzcd}
			\tilde{R}*\tilde{R}\ar[r,"d_{0}^{\tilde{R}}"]\ar[d,"d_{1}^{\tilde{R}}"'] & \tilde{R}\ar[d,"\tilde{r_{1}}"] \\
			\tilde{R}\ar[r,"\tilde{r_{0}}"'] & E
		\end{tikzcd}
	\end{center}
	There is a morphism $u\colon\tilde{R}*\tilde{R}\to R*R$ such that $d_{0}^{R}u=\bar{e}d_{0}^{\tilde{R}}$ and $d_{1}^{R}u=\bar{e}d_{1}^{\tilde{R}}$. Using the transitivity $\tau^{R}\colon R*R\to R$ of the pseudocongruence $R$, we then have the following cone over the diagram defining $\tilde{R}$:
	\begin{center}
		\begin{tikzcd}
			& \tilde{R}*\tilde{R}\ar[dl,dashed,"\tilde{r_0}d_{0}^{\tilde{R}}"']\ar[dr,dashed,"\tilde{r_1}d_{1}^{\tilde{R}}"]\ar[dd,dashed,"\tau^{R}u"'] & \\
			E\ar[dd,"e"'] & &  E\ar[dd,"e"] \\
			& R\ar[dl,"r_0"']\ar[dr,"r_1"] & \\
			X & & X
		\end{tikzcd}
	\end{center}
	Thus, there exists a $\tau^{\tilde{R}}\colon\tilde{R}*\tilde{R}\to\tilde{R}$ such that $\tilde{r_{0}}\tau^{\tilde{R}}=\tilde{r_{0}}d_{0}^{\tilde{R}}$, $\tilde{r_{1}}\tau^{\tilde{R}}=\tilde{r_{1}}d_{1}^{\tilde{R}}$ and $\bar{e}\tau^{\tilde{R}}=\tau^{R}u$.
	
	By construction, we have the morphism $\bar{e}\colon\tilde{R}\to R$ which shows that $e$ defines a morphism $[e]\colon(E,\tilde{R})\to(X,R)$ in $\ca{C}_{\rm{ex}}$. In addition, the morphism $\phi$ exhibits the inequality $[f][e]\leq[g][e]$.
	
	Now suppose that $[h]\colon(Z,T)\to(X,R)$ is such that $[f][h]\leq[g][h]$ in $\ca{C}_{\rm{ex}}$. This means that there exists a $\Sigma\colon Z\to S\in\ca{C}$ such that $s_{0}\Sigma=fh$ and $s_{1}\Sigma=gh$. By the definition of $E$ as a weak limit, there is an induced $u\colon Z\to E\in\ca{C}$ such that $eu=h$ and $\phi u=\Sigma$. Now we have the following cone
	\begin{center}
		\begin{tikzcd}
			& T\ar[dl,dashed,"ut_0"']\ar[dr,dashed,"ut_1"]\ar[dd,dashed,"\bar{h}"'] & \\
			E\ar[dd,"e"'] & &  E\ar[dd,"e"] \\
			& R\ar[dl,"r_0"']\ar[dr,"r_1"] & \\
			X & & X
		\end{tikzcd}
	\end{center}
	Hence, we obtain a $\bar{u}\colon T\to \tilde{R}$ such that $\tilde{r_0}\bar{u}=ut_0$, $\tilde{r_1}\bar{u}=ut_1$ and $\bar{e}\bar{u}=\bar{h}$. The first two of these equalities imply that $u$ defines a morphism $(Z,T)\to(E,\tilde{R})\in\ca{C}_{\rm{ex}}$ and then we trivially have $[e][u]=[h]$.
	
	Finally, consider $[u],[v]\colon (Z,T)\to(E,\tilde{R})$ in $\ca{C}_{\rm{ex}}$ such that $[e][u]\leq[e][v]$. So there exists a $\Sigma\colon Z\to R$ with $r_{0}\Sigma=eu$ and $r_{1}\Sigma=ev$. Then by the weak limit property of $\tilde{R}$, we get a $\tilde{\Sigma}\colon Z\to\tilde{R}$ such that $\tilde{r_0}\tilde{\Sigma}=u$, $\tilde{r_1}\tilde{\Sigma}=v$ and $\bar{e}\tilde{\Sigma}=\Sigma$. The first two equalities yield $[u]\leq[v]$.
\end{proof}

\begin{prop}\label{products in C_ex}
	$\ca{C}_{\rm{ex}}$ has finite products.
\end{prop}
\begin{proof}
	Consider objects $(X,R)$ and $(Y,S)$ in $\ca{C}_{\rm{ex}}$. Take a weak binary product $\begin{tikzcd}[cramped] X & P\ar[l,"\pi_X"']\ar[r,"\pi_Y"] & Y\end{tikzcd}$ in $\ca{C}$ and then form the following weak conical limit in $\ca{C}$
	\begin{center}
		\begin{tikzcd}
			& & & \Gamma\ar[dll,dashed,"\bar{\pi}_X"']\ar[drr,dashed,"\bar{\pi}_Y"]\ar[dddl,dashed,"\gamma_0"']\ar[dddr,dashed,"\gamma_1"] & & & \\
			& R\ar[dl,"r_0"']\ar[dr,"r_1"] & & & & S\ar[dl,"s_0"']\ar[dr,"s_1"] & \\
			X & & X & & Y & & Y  \\
			& & P\ar[ull,"\pi_X"]\ar[urr,"\pi_Y"',near start] & &  P\ar[ull,"\pi_X",near start]\ar[urr,"\pi_Y"'] & & 
		\end{tikzcd}
	\end{center}
	We claim that $\begin{tikzcd}[cramped] \Gamma\ar[r,shift left=0.75ex,"\gamma_0"]\ar[r,shift right=0.75ex,"\gamma_1"'] & P
	\end{tikzcd}$ is a pseudocongruence in $\ca{C}$ and that then \\
	$\begin{tikzcd}[cramped] (X,R) & (P,\Gamma)\ar[l,"{[\pi_X]}"']\ar[r,"{[\pi_Y]}"] & (Y,S)
	\end{tikzcd}$ is  a product diagram in $\ca{C}_{\rm{ex}}$.First, consider any $a_{0}\leq a_{1}\colon A\to P$ in $\ca{C}$. Then we have both $\pi_{X}a_{0}\leq \pi_{X}a_{1}$ and $\pi_{Y}a_{0}\leq \pi_{Y}a_{1}$, so by order-reflexivity of $R$ and $S$ we get morphisms $u\colon A\to R$ and $v\colon A\to S$ such that $r_{0}u=\pi_{X}a_{0}$, $r_{1}u=\pi_{X}a_{1}$ and $s_{0}v=\pi_{Y}a_{0}$, $s_{1}v=\pi_{Y}a_{1}$. Now 
	
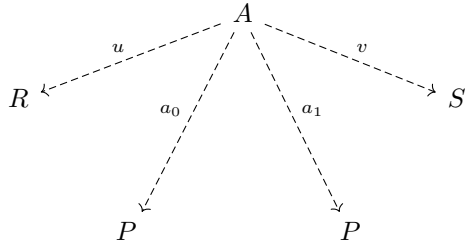
\begin{center}
		\begin{tikzcd}
			& & & A\ar[dll,dashed,"u"']\ar[drr,dashed,"v"]\ar[dddl,dashed,"a_0"']\ar[dddr,dashed,"a_1"] & & & \\		& R & & & & S & \\
			& &  & &  & &   \\
			& & P\ & &  P & & 
		\end{tikzcd}
	\end{center}
	forms a cone over the above diagram, hence there is a $w\colon A\to\Gamma$ such that $\pi_{R}w=u$, $\pi_{S}w=v$, $\gamma_{0}w=a_0$ and $\gamma_{1}w=a_{1}$.
	
	To prove transitivity of $\Gamma$, observe first that we have $r_{0}(\bar{\pi}_X d_{1}^{\Gamma})=\pi_{X}\gamma_{0}d_{1}^{\Gamma}=\pi_{X}\gamma_{1}d_{0}^{\Gamma}=r_{1}(\bar{\pi}_X d_{0}^{\Gamma})$, and so there is a $u_{R}\colon\Gamma*\Gamma\to R*R$ making the following commute:
	
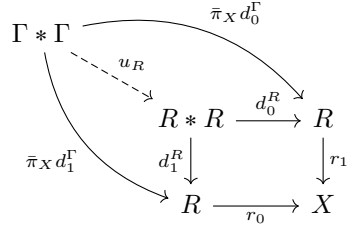
\begin{center}
		\begin{tikzcd}
			\Gamma*\Gamma\ar[dr,dashed,"u_R"]\ar[ddr,"\bar{\pi}_X d_{1}^{\Gamma}"',bend right]\ar[drr,"\bar{\pi}_X d_{0}^{\Gamma}",bend left] & & \\
			& R*R\ar[r,"d_{0}^{R}"]\ar[d,"d_{1}^{R}"'] & R\ar[d,"r_1"] \\
			& R\ar[r,"r_0"'] & X
		\end{tikzcd}
	\end{center}
	Similarly, we get $u_{S}\colon \Gamma*\Gamma\to S*S$ 
	
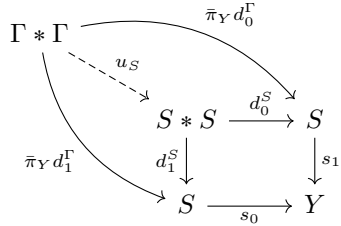
\begin{center}
		\begin{tikzcd}
			\Gamma*\Gamma\ar[dr,dashed,"u_S"]\ar[ddr,"\bar{\pi}_Y d_{1}^{\Gamma}"',bend right]\ar[drr,"\bar{\pi}_Y d_{0}^{\Gamma}",bend left] & & \\
			& S*S\ar[r,"d_{0}^{S}"]\ar[d,"d_{1}^{S}"'] & S\ar[d,"s_1"] \\
			& S\ar[r,"s_0"'] & Y
		\end{tikzcd}
	\end{center}
	Now we have the following cone over our initial diagram defining $\Gamma$:
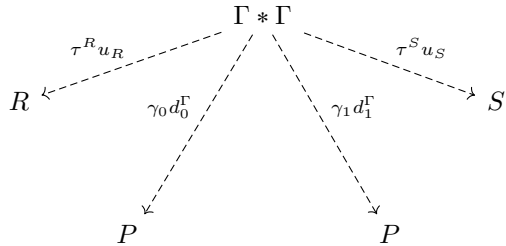
	\begin{center}
		\begin{tikzcd}
			& & & \Gamma*\Gamma\ar[dll,dashed,"\tau^{R}u_R"']\ar[drr,dashed,"\tau^{S}u_S"]\ar[dddl,dashed,"\gamma_{0}d_{0}^{\Gamma}"']\ar[dddr,dashed,"\gamma_{1}d_{1}^{\Gamma}"] & & & \\		& R & & & & S & \\
			& &  & &  & &   \\
			& & P\ & &  P & & 
		\end{tikzcd}
	\end{center}
	Thus, there exists a $\tau^{\Gamma}\colon \Gamma*\Gamma\to\Gamma$ such that $\bar{\pi}_X\tau^{\Gamma}=\tau^{R}u_R$, $\bar{\pi}_Y\tau^{\Gamma}=\tau^{S}u_S$, $\gamma_{0}\tau^{\Gamma}=\gamma_{0}d_{0}^{\Gamma}$ and $\gamma_{1}\tau^{\Gamma}=\gamma_{1}d_{1}^{\Gamma}$. In particular, $\tau^{\Gamma}$ exhibits transitivity of $\Gamma$.
	
	Next, consider any pair of morphisms $\begin{tikzcd}[cramped] (X,R) & (Z,T)\ar[l,"{[f]}"']\ar[r,"{[g]}"] & (Y,S)
	\end{tikzcd}$ in $\ca{C}_{\rm{ex}}$. First of all, by the weak product property of $P$, we have an $h\colon Z\to P$ with $\pi_{X}h=f$ and $\pi_{Y}h=g$. Then we have the following cone over the diagram defining $\Gamma$ in $\ca{C}$,
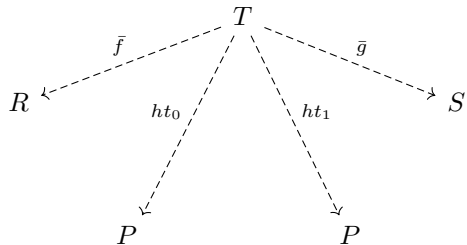
	\begin{center}
		\begin{tikzcd}
			& & & T\ar[dll,dashed,"\bar{f}"']\ar[drr,dashed,"\bar{g}"]\ar[dddl,dashed,"ht_0"']\ar[dddr,dashed,"ht_1"] & & & \\		& R & & & & S & \\
			& &  & &  & &   \\
			& & P\ & &  P & & 
		\end{tikzcd}
	\end{center}
	simply because $[f],[g]$ are morphisms in $\ca{C}_{\rm{ex}}$. Thus, we obtain an $\bar{h}\colon T\to\Gamma$ such that $\bar{\pi}_X\bar{h}=\bar{f}$, $\bar{\pi}_Y\bar{h}=\bar{g}$, $\gamma_{0}\bar{h}=ht_0$ and $\gamma_{1}\bar{h}=ht_1$. The last two equalities say precisely that $h$ defines a morphism $[h]\colon(Z,T)\to(P,\Gamma)$ and then we immediately have $[\pi_{X}]\circ[h]=[f]$ and $[\pi_{Y}]\circ[h]=[g]$.
	
	Finally, consider arrows $\begin{tikzcd}[cramped](Z,T)\ar[r,"{[h]}",shift left=0.75ex]\ar[r,"{[h']}"',shift right=0.75ex] & (P,\Gamma) \end{tikzcd}$ such that $[\pi_{X}]\circ[h]\leq[\pi_{X}]\circ[h']$ and $[\pi_{Y}]\circ[h]\leq[\pi_{Y}]\circ[h']$. This means that there are $\Sigma_{R}\colon Z\to R$ with $r_{0}\Sigma_{R}=\pi_{X}h$, $r_{1}\Sigma_{R}=\pi_{X}h'$ and $\Sigma_{S}\colon Z\to S$ with $s_{0}\Sigma_{S}=\pi_{Y}h$, $s_{1}\Sigma_{S}=\pi_{Y}h'$. Then again we have a cone
	
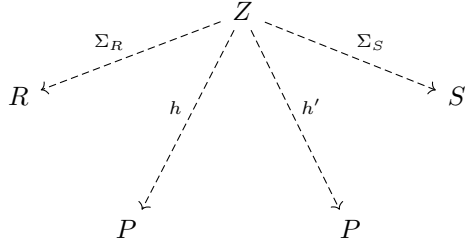
\begin{center}
		\begin{tikzcd}
			& & & Z\ar[dll,dashed,"\Sigma_R"']\ar[drr,dashed,"\Sigma_S"]\ar[dddl,dashed,"h"']\ar[dddr,dashed,"{h'}"] & & & \\		& R & & & & S & \\
			& &  & &  & &   \\
			& & P\ & &  P & & 
		\end{tikzcd}
	\end{center}
	and so we get a $\Sigma\colon Z\to\Gamma$ such that $\bar{\pi}_X\Sigma=\Sigma_{R}$, $\bar{\pi}_Y\Sigma=\Sigma_{S}$, $\gamma_{0}\Sigma=h$ and $\gamma_{1}\Sigma=h'$. The last two equalities yield $[h]\leq[h']$ in $\ca{C}_{\rm{ex}}$.
	\vspace{1cm}
	
	It remains to construct a terminal object for $\ca{C}_{\rm{ex}}$. Let $T\in\ca{C}$ be a weak terminal object and take a weak binary product $\begin{tikzcd}[cramped] T & T\times T\ar[l,"\pi_{0}"']\ar[r,"\pi_1"] & T \end{tikzcd}$ in $\ca{C}$. Then it is easy to see that $\begin{tikzcd}[cramped] T\times T\ar[r,"\pi_0",shift left=0.75ex]\ar[r,"\pi_1"',shift right=0.75ex] & T \end{tikzcd}$ is a pseudocongruence in $\ca{C}$. For any $(X,R)\in\ca{C}_{\rm{ex}}$, we first have a morphism $f\colon X\to T\in\ca{C}$, and then by the weak product property an $\bar{f}\colon R\to T\times T$ with $\pi_{0}\bar{f}=fr_0$ and $\pi_{1}\bar{f}=fr_1$. So we have a morphism $[f]\colon(X,R)\to(T,T\times T)$. If $[g]$ is any other such morphism, then the weak product gives a $\Sigma\colon X\to T\times T$ with $\pi_{0}\Sigma=f$ and $\pi_{1}\Sigma=g$, so that $[f]\leq[g]$. Symmetrically, we also have $[g]\leq[f]$.
\end{proof}

Combining \ref{inserters in C_ex} and \ref{products in C_ex}, we have that $\ca{C}_{\rm{ex}}$ has all finite limits. 

\begin{rmk}\label{commas and pullbacks in C_ex}
	Even though we can in principle construct any finite weighted limit via finite products and inserters, it will be useful for what follows to have a direct and explicit description of two particular types of such limits.
	\begin{enumerate}
		\item \underline{Comma squares}: To form a comma square of the form
		\begin{center}
			\begin{tikzcd}
				(C,\Gamma)\ar[r,"{[\pi_{Y}]}"]\ar[d,"{[\pi_{X}]}"']\ar[dr,phantom,"\leq"] & (Y,S)\ar[d,"{[g]}"] \\
				(X,R)\ar[r,"{[f]}"'] & (Z,T)
			\end{tikzcd}
		\end{center}
		in $\ca{C}_{\rm{ex}}$, we successively take the following weak conical limits in $\ca{C}$:
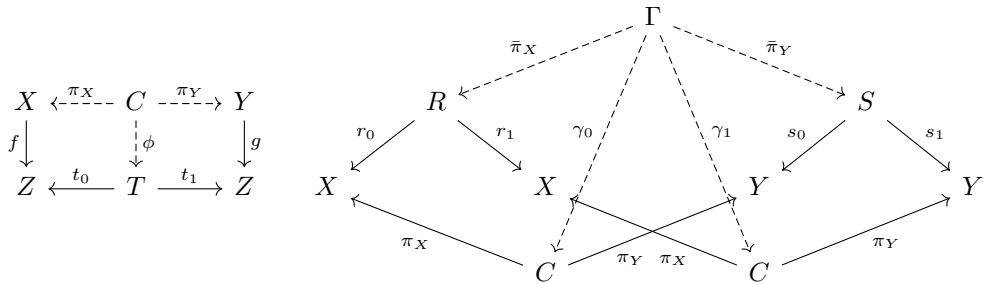
		\begin{center}
			\begin{tikzcd}
				X\ar[d,"f"'] & C\ar[l,dashed,"\pi_X"']\ar[r,dashed,"\pi_Y"]\ar[d,dashed,"\phi"] & Y\ar[d,"g"] \\
				Z & T\ar[l,"t_0"']\ar[r,"t_1"] & Z
			\end{tikzcd}
			\quad
			\begin{tikzcd}
				& & & \Gamma\ar[dll,dashed,"\bar{\pi}_X"']\ar[drr,dashed,"\bar{\pi}_Y"]\ar[dddl,dashed,"\gamma_0"']\ar[dddr,dashed,"\gamma_1"] & & & \\
				& R\ar[dl,"r_0"']\ar[dr,"r_1"] & & & & S\ar[dl,"s_0"']\ar[dr,"s_1"] & \\
				X & & X & & Y & & Y  \\
				& & C\ar[ull,"\pi_X"]\ar[urr,"\pi_Y"',near start] & &  C\ar[ull,"\pi_X",near start]\ar[urr,"\pi_Y"'] & & 
			\end{tikzcd}
		\end{center}
		The fact that $\begin{tikzcd}[cramped] \Gamma\ar[r,shift left=0.75ex,"\gamma_0"]\ar[r,shift right=0.75ex,"\gamma_1"'] & P
		\end{tikzcd}$ is a pseudocongruence in $\ca{C}$ follows exactly as in the proof of \ref{products in C_ex}. Indeed, note that said proof relied solely on the shape of the diagram defining $\Gamma$ and the fact that $R$ and $S$ are pseudocongruences. The inequality $[f]\circ[\pi_X]\leq[g]\circ[\pi_Y]$ is exhibited by the morphism $\phi$ and the commutativity of the first weak limit diagram above. 
		
		Now let $[x]\colon (A,V)\to(X,R)$ and $[y]\colon (A,V)\to(Y,S)$ be such that $[f]\circ[x]\leq[g]\circ[y]$ in $\ca{C}_{\rm{ex}}$. This means that there is a $\Sigma\colon A\to T$ such that $t_{0}\Sigma=fx$ and $t_{1}\Sigma=gy$. By the property of the first weak limit above, we obtain a $u\colon A\to C$ with $\pi_{X}u=x$, $\pi_{Y}u=y$ and $\phi u=\Sigma$. We then have the following cone over the second weak limit diagram above
		
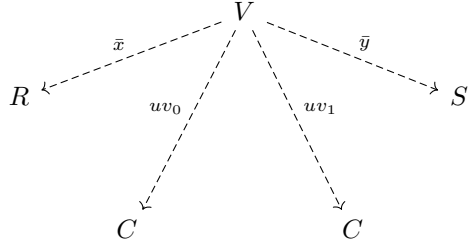
\begin{center}
			\begin{tikzcd}
				& & & V\ar[dll,dashed,"\bar{x}"']\ar[drr,dashed,"\bar{y}"]\ar[dddl,dashed,"uv_0"']\ar[dddr,dashed,"uv_1"] & & & \\		& R & & & & S & \\
				& &  & &  & &   \\
				& & C\ & &  C & & 
			\end{tikzcd}
		\end{center}
		Thus, there is a $\bar{u}\colon V\to\Gamma$ such that $\bar{\pi}_X\bar{u}=\bar{x}$, $\bar{\pi}_Y\bar{u}=\bar{y}$, $\gamma_{0}\bar{u}=uv_0$ and $\gamma_{1}\bar{u}=uv_1$. The latter two equalities show that $u$ defines a morphism $[u]\colon (A,V)\to(C,\Gamma)$ in $\ca{C}_{\rm{ex}}$ and then we immediately have $[\pi_{X}]\circ[u]=[x]$ and $[\pi_{Y}]\circ[u]=[y]$.
		
		Finally, consider arrows $\begin{tikzcd}[cramped](A,V)\ar[r,"{[u]}",shift left=0.75ex]\ar[r,"{[u']}"',shift right=0.75ex] & (C,\Gamma) \end{tikzcd}$ such that $[\pi_{X}]\circ[u]\leq[\pi_{X}]\circ[u']$ and $[\pi_{Y}]\circ[u]\leq[\pi_{Y}]\circ[u']$. 
		This means that there are $\Sigma_{R}\colon A\to R$ with $r_{0}\Sigma_{R}=\pi_{X}u$, $r_{1}\Sigma_{R}=\pi_{X}u'$ and $\Sigma_{S}\colon A\to S$ with $s_{0}\Sigma_{S}=\pi_{Y}u$, $s_{1}\Sigma_{S}=\pi_{Y}u'$. Then again we have a cone
		
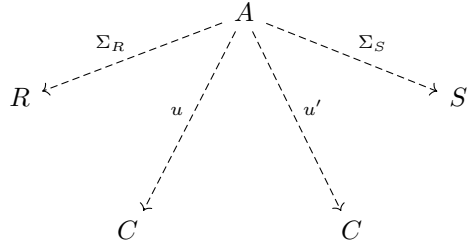
\begin{center}
			\begin{tikzcd}
				& & & A\ar[dll,dashed,"\Sigma_R"']\ar[drr,dashed,"\Sigma_S"]\ar[dddl,dashed,"u"']\ar[dddr,dashed,"{u'}"] & & & \\		& R & & & & S & \\
				& &  & &  & &   \\
				& & C\ & &  C & & 
			\end{tikzcd}
		\end{center}
		and so we get a $\Sigma\colon A\to\Gamma$ such that $\bar{\pi}_X\Sigma=\Sigma_{R}$, $\bar{\pi}_Y\Sigma=\Sigma_{S}$, $\gamma_{0}\Sigma=u$ and $\gamma_{1}\Sigma=u'$. The last two equalities say that $[u]\leq[u']$ in $\ca{C}_{\rm{ex}}$.
		\item \underline{Pullbacks}: The construction of a pullback square
		\begin{center}
			\begin{tikzcd}
				(P,\Gamma)\ar[r,"{[\pi_{Y}]}"]\ar[d,"{[\pi_{X}]}"']\ar[dr,phantom,"\leq"] & (Y,S)\ar[d,"{[g]}"] \\
				(X,R)\ar[r,"{[f]}"'] & (Z,T)
			\end{tikzcd}
		\end{center}
		proceeds almost verbatim as the construction of the comma square above, except that in the first step we define $P$ as a weak conical limit as follows:
		
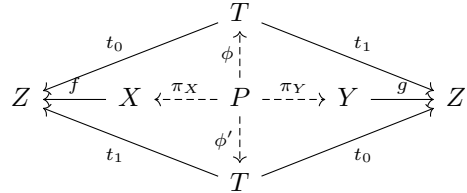
\begin{center}
			\begin{tikzcd}
				& & T\ar[dll,"t_0"']\ar[drr,"t_1"] & & \\
				Z & X\ar[l,"f"'] & P\ar[u,dashed,"\phi"]\ar[d,dashed,"\phi'"']\ar[r,dashed,"\pi_Y"]\ar[l,dashed,"\pi_X"'] & Y\ar[r,"g"] & Z \\
				& & T\ar[ull,"t_1"]\ar[urr,"t_0"'] & & 
			\end{tikzcd}
		\end{center}
		The required verification is then entirely analogous.
	\end{enumerate}
\end{rmk}

\begin{prop}\label{(so,ff) factorizations in C_ex}
	$\ca{C}_{\rm{ex}}$ has $(\textrm{effective epi},\textrm{order-mono})$ factorizations.
\end{prop}
\begin{proof}
	Consider any $[f]\colon(X,R)\to(Y,S)\in\ca{C}_{\rm{ex}}$. Form the following weak conical limit in $\ca{C}$
	\begin{center}
		\begin{tikzcd}
			X\ar[d,"f"'] & I\ar[l,dashed,"i_0"']\ar[r,dashed,"i_1"]\ar[d,dashed,"\phi"] & X\ar[d,"f"] \\
			Y & S\ar[l,"s_0"']\ar[r,"s_1"] & Y
		\end{tikzcd}
	\end{center}
	First of all, $\begin{tikzcd}[cramped] I\ar[r,"i_0",shift left=0.75ex]\ar[r,"i_1"',shift right=0.75ex] & X
	\end{tikzcd}$ is a pseudocongruence in $\ca{C}$:
	Let $a_{0},a_{1}\colon A\to X\in \ca{C}$ with $a_0\leq a_1$. Then $fa_0\leq fa_1$, so there is a $u\colon A\to S$ with $s_{0}u=fa_0$ and $s_{1}u=fa_1$. Thus, by the weak limit property, there exists a $u'\colon A\to I$ such that $i_{0}u'=a_0$, $i_{1}u'=a_1$ and $\phi u'=u$.
	
	For transitivity, observe that $s_{0}\phi d_{1}^{I}=fi_{0}d_{1}^{I}=fi_{1}d_{0}^{I}=s_{1}\phi d_{0}^{I}$, so there is a $\bar{\phi}\colon I*I\to S*S$ such that $\phi d_{0}^{I}=d_{0}^{S}\bar{\phi}$ and $\phi d_{1}^{I}=d_{1}^{S}\bar{\phi}$.
	\vspace{0.5cm}
	
	Now we have a factorization $[f]=\begin{tikzcd}[cramped] (X,R)\ar[r,"{[1_X]}"] & (X,I)\ar[r,"{[f]}"] & (Y,S)
	\end{tikzcd}$ in $\ca{C}_{\rm{ex}}$. To see that $1_X$ indeed defines a morphism as indicated, use the $\bar{f}\colon R\to S$ which is furnished by the fact that $[f]$ is a morphism $(X,R)\to(Y,S)$. Then the weak limit property of $I$ yields a $\psi\colon R\to I$ with $i_{0}\psi=r_0$, $i_{1}\psi=r_1$ and $\phi\psi=\bar{f}$. The fact that $f$ also defines a morphism $(X,I)\to(Y,S)$ is exhibited by the morphism $\phi$.
	
	We claim that $\begin{tikzcd}[cramped] (X,I)\ar[r,"{[f]}"] & (Y,S)  \end{tikzcd}$ is an order mono and that $\begin{tikzcd}[cramped] (X,R)\ar[r,"{[1_X]}"] & (X,I)  \end{tikzcd}$ is an effective epi.
	
	For the first claim, consider a pair $\begin{tikzcd}[cramped] (Z,T)\ar[r,"{[u]}",shift left=0.75ex]\ar[r,"{[v]}"',shift right=0.75ex] & (X,I)
	\end{tikzcd}$ with $[f]\circ[u]\leq[f]\circ[v]$. Hence, there is a $\Sigma\colon Z\to S$ such that $s_{0}\Sigma=fu$ and $s_{1}\Sigma=fv$. Then by the weak limit property of $I$ we get a $\Sigma'\colon Z\to I$ with $i_{0}\Sigma'=u$, $i_{1}\Sigma'=v$ and $\phi\Sigma'=\Sigma$, so that $[u]\leq[v]$.
	
	To show that $[1_X]\colon (X,R)\to(X,I)$ is an effective epi, we will show that it (as it must be) is the coinserter of its kernel congruence. To form the latter, we follow the construction given in \ref{commas and pullbacks in C_ex}. In the present situation, since both underlying morphisms are $1_X$, the weak limit $C$ is actually $I$ itself. Thus, the kernel congruence is of the form $\begin{tikzcd}[cramped] (I,\Gamma)\ar[r,"{[i_0]}",shift left=0.75ex]\ar[r,"{[i_1]}"',shift right=0.75ex] & (X,R)
	\end{tikzcd}$. 
	
	Now let $[g]\colon (X,R)\to(Z,T)$ be such that $[g]\circ[i_0]\leq[g]\circ[i_1]$. Then there is a $\Sigma\colon I\to T$ with $t_{0}\Sigma=gi_0$ and $t_{1}\Sigma=gi_1$. But these equalities now imply that $g$ also defines a morphism $(X,I)\to(Z,T)$, which is clearly the desired factorization.
	\begin{center}
		\begin{tikzcd} (I,\Gamma)\ar[r,"{[i_0]}",shift left=0.75ex]\ar[r,"{[i_1]}"',shift right=0.75ex] & (X,R)\ar[r,"{[1_X]}"]\ar[dr,"{[g]}"'] & (X,I)\ar[d,dashed,"{[g]}"] \\
			& & (Z,T)
		\end{tikzcd}
	\end{center}
	
	Finally, if $\begin{tikzcd}[cramped] (X,I)\ar[r,"{[g]}",shift left=0.75ex]\ar[r,"{[h]}"',shift right=0.75ex] & (Z,T)
	\end{tikzcd}$ are such that $[g]\circ[1_X]\leq[h]\circ[1_X]$, then there is a $\Sigma'\colon X\to T$ with $t_{0}\Sigma'=g$ and $t_{1}\Sigma'=h$, which is to say that $[g]\leq[h]$ in $\ca{C}_{\rm{ex}}$.
\end{proof}

\begin{prop}
	Effective epis are stable under pullback in $\ca{C}_{\rm{ex}}$.
\end{prop}
\begin{proof}
	By the proof of \ref{(so,ff) factorizations in C_ex}, we know that every effective epi in $\ca{C}_{\rm{ex}}$ is isomorphic to one whose underlying morphism in $\ca{C}$ is an identity. Thus, it suffices to consider a pullback square of the form
	\begin{center}
		\begin{tikzcd}
			(P,\Gamma)\ar[r,"{[\pi_Y]}"]\ar[d,"{[\pi_X]}"'] & (Y,S)\ar[d,"{[g]}"] \\
			(X,R)\ar[r,"{[1_X]}"'] & (X,T)
		\end{tikzcd}
	\end{center}
	Following the construction of pullbacks given in \ref{commas and pullbacks in C_ex}, we must successively take the following weak conical limits in $\ca{C}$:
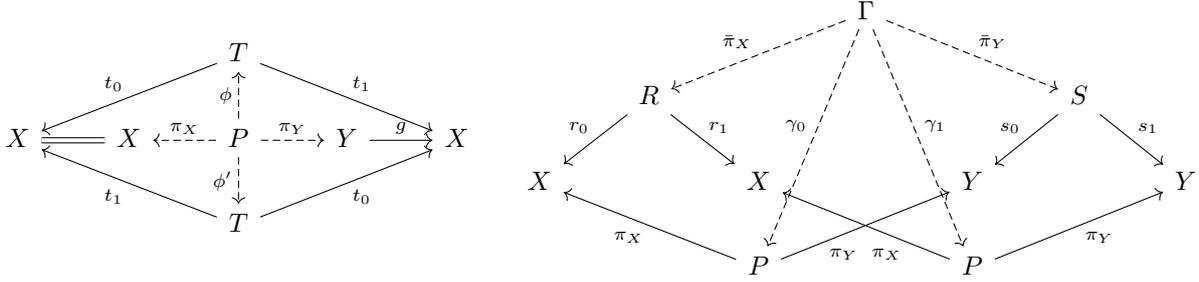
	\begin{center}
		\begin{tikzcd}
			& & T\ar[dll,"t_0"']\ar[drr,"t_1"] & & \\
			X & X\ar[l,equal] & P\ar[u,dashed,"\phi"]\ar[d,dashed,"\phi'"']\ar[r,dashed,"\pi_Y"]\ar[l,dashed,"\pi_X"'] & Y\ar[r,"g"] & X \\
			& & T\ar[ull,"t_1"]\ar[urr,"t_0"'] & & 
		\end{tikzcd}
		\quad
		\begin{tikzcd}
			& & & \Gamma\ar[dll,dashed,"\bar{\pi}_X"']\ar[drr,dashed,"\bar{\pi}_Y"]\ar[dddl,dashed,"\gamma_0"']\ar[dddr,dashed,"\gamma_1"] & & & \\
			& R\ar[dl,"r_0"']\ar[dr,"r_1"] & & & & S\ar[dl,"s_0"']\ar[dr,"s_1"] & \\
			X & & X & & Y & & Y  \\
			& & P\ar[ull,"\pi_X"]\ar[urr,"\pi_Y"',near start] & &  P\ar[ull,"\pi_X",near start]\ar[urr,"\pi_Y"'] & & 
		\end{tikzcd}
	\end{center}
	We will show that the order-monic part of the factorization is an iso. Recall by the previous proposition how this is constructed: take a weak conical limit 
	\begin{center}
		\begin{tikzcd}
			P\ar[d,"\pi_Y"'] & I\ar[l,dashed,"i_0"']\ar[r,dashed,"i_1"]\ar[d,dashed,"{\phi''}"] & P\ar[d,"\pi_Y"] \\
			Y & S\ar[l,"s_0"]\ar[r,"s_1"'] & Y
		\end{tikzcd}
	\end{center}
	Then the factorization is given by $[\pi_Y]=\begin{tikzcd}[cramped] (P,\Gamma)\ar[r,two heads,"{[1_P]}"] & (P,I)\ar[r,tail,"{[\pi_Y]}"] & (Y,S)\end{tikzcd}$. We will now show that $[\pi_Y]\colon (P,I)\rightarrowtail(Y,S)$ is also a split epi.
	
	First, observe that we have the following cone in $\ca{C}$,
	\begin{center}
		\begin{tikzcd}
			& T &  \\
			X &  Y\ar[u,dashed,"\bar{g}\delta^{S}"]\ar[d,dashed,"\bar{g}\delta^{S}"']\ar[r,dashed,"1_Y"]\ar[l,dashed,"g"'] & Y \\
			& T & 
		\end{tikzcd}
	\end{center}
	where $\delta^{S}$ is a common splitting of $s_0$,$s_1$ (furnished by reflexivity). Thus, there is an $m\colon Y\to P$ such that $\phi m=\bar{g}\delta^{S}=\phi'm$, $\pi_{X}m=g$ and $\pi_{Y}m=1_Y$. So in turn it suffices to show that $m$ defines a morphism $(Y,S)\to(P,I)$ in $\ca{C}_{\rm{ex}}$. But since $\pi_{Y}(ms_0)=s_0$ and $\pi_{Y}(ms_1)=s_1$, by the definition of $I$ as a weak limit we get an $\bar{m}\colon S\to I$ such that $\phi''\bar{m}=1_S$, $i_{0}\bar{m}=ms_0$ and $i_{1}\bar{m}=ms_1$. This completes the proof.
\end{proof}

\begin{prop}\label{Effective Congruences in C_ex}
	Congruences in $\ca{C}_{\rm{ex}}$ are effective.
\end{prop}
\begin{proof}
	Consider a congruence $\begin{tikzcd}[cramped] (X,R)\ar[r,shift left=0.75ex,"{[h_0]}"]\ar[r,shift right=0.75ex,"{[h_1]}"'] & (Y,S) \end{tikzcd}$ in $\ca{C}_{\rm{ex}}$. Let us begin with analyzing what this entails in terms of arrows in $\ca{C}$.
	
	By the construction of comma squares given in \ref{commas and pullbacks in C_ex}, the comma of the pair of identity morphisms on $(Y,S)$ in $\ca{C}_{\rm{ex}}$ is precisely
	\begin{center}
		\begin{tikzcd}
			(S,\Gamma)\ar[r,"{[s_1]}"]\ar[d,"{[s_0]}"']\ar[dr,phantom,"\leq"] & (Y,S)\ar[d,equal] \\
			(Y,S)\ar[r,equal] & (Y,S)
		\end{tikzcd}
	\end{center} 
	So order-reflexivity of the congruence $([h_0],[h_1])$ means that there exists a morphism $[u]\colon(S,\Gamma)\to(X,R)$ such that $[h_0]\circ[u]=[s_0]$ and $[h_1]\circ[u]=[s_1]$. From these two equalities we respectively get morphisms $\Sigma_{0},\Sigma_{0}'\colon S\to S$ and $\Sigma_{1},\Sigma_{1}'\colon S\to S$ such that the following equalities hold:
	\begin{equation}\label{order-reflexivity equations 1}
		\begin{aligned}
			s_0\Sigma_0&=h_0u & s_0\Sigma_{0}'&=s_0 \\
			s_1\Sigma_0&=s_0       & s_1\Sigma_{0}'&=h_{0}u
		\end{aligned}
	\end{equation}
	\begin{equation}\label{order-reflexivity equations 2}
		\begin{aligned}
			s_0\Sigma_1&=h_1u & s_0\Sigma_{1}'&=s_1 \\
			s_1\Sigma_1&=s_1       & s_1\Sigma_{1}'&=h_{1}u
		\end{aligned}
	\end{equation}
	
	For transitivity, recall first that the pullback square
	\begin{center}
		\begin{tikzcd}
			(P,\Gamma)\ar[r,"{[\pi_0]}"]\ar[d,"{[\pi_1]}"'] & (X,R)\ar[d,"{[h_1]}"] \\
			(X,R)\ar[r,"{[h_0]}"'] & (Y,S)
		\end{tikzcd}
	\end{center}
	is constructed by first taking the following weak conical limit in $\ca{C}$:
	\begin{center}
		\begin{tikzcd}
			& & S\ar[dll,"s_0"']\ar[drr,"s_1"] & & \\
			Y & X\ar[l,"h_0"'] & P\ar[u,dashed,"\phi"]\ar[d,dashed,"\phi'"']\ar[r,dashed,"\pi_0"]\ar[l,dashed,"\pi_1"'] & X\ar[r,"h_1"] & Y \\
			& & S\ar[ull,"s_1"]\ar[urr,"s_0"'] & & 
		\end{tikzcd}
	\end{center}
	Now transitivity of $([h_0],[h_1])$ in $\ca{C}_{\rm{ex}}$ means that there exists a $[\tau]\colon (P,\Gamma)\to(X,R)$ such that $[h_0]\circ[\tau]=[h_0]\circ[\pi_0]$ and $[h_1]\circ[\tau]=[h_1]\circ[\pi_1]$. Hence, we respectively get morphisms $\psi_{0},\psi_{0}'\colon P\to S$ and  $\psi_{1},\psi_{1}'\colon P\to S$ such that
	\begin{equation}
		\begin{aligned}
			s_0\psi_0&=h_0\tau & s_0\psi_{0}'&=h_0\pi_0 \\
			s_1\psi_0&=h_0\pi_0       & s_1\psi_{0}'&=h_{0}\tau
		\end{aligned}
	\end{equation}
	\begin{equation}
		\begin{aligned}
			s_0\psi_1&=h_1\tau & s_0\psi_{1}'&=h_1\pi_1 \\
			s_1\psi_1&=h_1\pi_1       & s_1\psi_{1}'&=h_{1}\tau
		\end{aligned}
	\end{equation}
	\vspace{0.5cm}
	
	Form the following weak conical limit in $\ca{C}$:
	\begin{center}
		\begin{tikzcd}
			& S\ar[dl,"s_1"']\ar[d,"s_0"'] & S'\ar[l,dashed,"p_0"']\ar[ddl,dashed,"{p_{0}'}"',near start]\ar[d,dashed,"{s'}"]\ar[ddr,dashed,"{p_{1}'}",near start]\ar[r,dashed,"p_1"] & S\ar[d,"s_0"]\ar[dr,"s_1"] &  \\
			Y & Y & X\ar[l,"h_0"']\ar[r,"h_1"] & Y & Y \\
			& S\ar[ul,"s_0"]\ar[u,"s_1"] & & S\ar[u,"s_1"']\ar[ur,"s_0"'] & 
		\end{tikzcd}
	\end{center}
	We set $s_{0}'\coloneqq s_{0}p_{0}'=s_{1}p_0$ and $s_{1}'\coloneqq s_{1}p_1=s_{0}p_{1}'$. Now the claim, which will occupy most of this proof, is that $\begin{tikzcd}[cramped] S'\ar[r,shift left=0.75ex,"{s_0'}"]\ar[r,shift right=0.75ex,"{s_1'}"'] & Y\end{tikzcd}$ is a pseudocongruence on $Y$ in $\ca{C}$. Before proving this though, let us observe first that there is a cone over the diagram above
	\begin{center}
		\begin{tikzcd}
			S & S\ar[l,dashed,"\Sigma_0"']\ar[ddl,dashed,"{\Sigma_{0}'}"']\ar[d,dashed,"u"]\ar[ddr,dashed,"{\Sigma_{1}'}"]\ar[r,dashed,"\Sigma_1"] & S   \\
			& X &  \\
			S & & S 
		\end{tikzcd}
	\end{center}
	where all needed equalities are contained in \ref{order-reflexivity equations 1}, \ref{order-reflexivity equations 2}. Thus, there exists a $q\colon S\to S'$ such that $p_{0}q=\Sigma_{0}$, $p_{0}'q=\Sigma_{0}'$, $p_{1}'q=\Sigma_{1}'$ and $p_{1}q=\Sigma_{1}$. Then we have $s_{0}'q=s_{0}p_{0}'q=s_{0}\Sigma_{0}'=s_{0}$ and $s_{1}'q=s_{1}p_{1}q=s_{1}\Sigma_{1}=s_1$, which is to say that $1_Y$ will define a morphism $(Y,S)\to(Y,S')$ in $\ca{C}_{\rm{ex}}$. In fact, this will be the coinserter of the congruence $([h_0],[h_1])$ and so to show that the latter is effective we will show that it is the kernel congruence of said coinserter.
	
	In wading through the arguments below, it could be useful to the reader to bear in mind the following set-theoretic intuition for $S'$: as a relation on $Y$, $S'$ consists by definition of those pairs $(y,y')$ for which there exists an $x\in X$ such that the following hold
	\begin{displaymath}
		(h_0(x),y)\in S,\quad (y,h_0(x))\in S,\quad (y',h_1(x))\in S,\quad (h_1(x),y')\in S
	\end{displaymath}
	These are to say that in the quotient of $X$ by the preorder $S$ we have that $y$ becomes equal to some $h_0(x)$ and $y'$ becomes equal to the corresponding $h_1(x)$.
	\vspace{0.5cm}
	
	Let us start with the technically most involved part, which is showing that $S'$ is transitive. In the following two diagrams we have taken weak pullbacks $S*S$ and $S'*S'$ in $\ca{C}$ and it is easy to check that the solid parts commute. This then gives the induced dashed arrows.
	
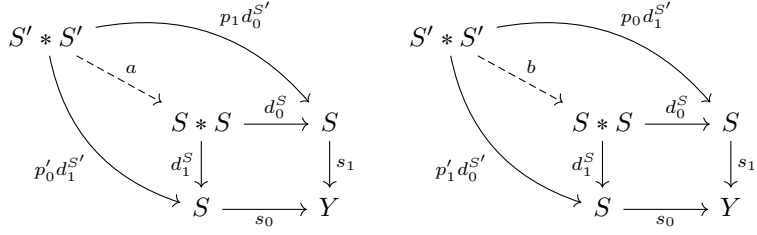
\begin{center}
		\begin{tikzcd}
			S'*S'\ar[dr,dashed,"a"]\ar[ddr,"{p_{0}'d_{1}^{S'}}"',bend right]\ar[drr,"p_{1}d_{0}^{S'}",bend left] & & \\
			& S*S\ar[r,"d_{0}^{S}"]\ar[d,"d_{1}^{S}"'] & S\ar[d,"s_1"] \\
			& S\ar[r,"s_0"'] & Y
		\end{tikzcd}
		\quad
		\begin{tikzcd}
			S'*S'\ar[dr,dashed,"b"]\ar[ddr,"{p_{1}'d_{0}^{S'}}"',bend right]\ar[drr,"p_{0}d_{1}^{S'}",bend left] & & \\
			& S*S\ar[r,"d_{0}^{S}"]\ar[d,"d_{1}^{S}"'] & S\ar[d,"s_1"] \\
			& S\ar[r,"s_0"'] & Y
		\end{tikzcd}
	\end{center}
	Now we have the following cone over the diagram defining $P$
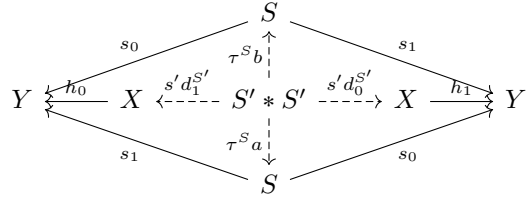
	\begin{center}
		\begin{tikzcd}
			& & S\ar[dll,"s_0"']\ar[drr,"s_1"] & & \\
			Y & X\ar[l,"h_0"'] & S'*S'\ar[u,dashed,"\tau^{S}b"]\ar[d,dashed,"\tau^{S}a"']\ar[r,dashed,"{s'd_{0}^{S'}}"]\ar[l,dashed,"{s'd_{1}^{S'}}"'] & X\ar[r,"h_1"] & Y \\
			& & S\ar[ull,"s_1"]\ar[urr,"s_0"'] & & 
		\end{tikzcd}
	\end{center} 
	Thus, there exists a $c\colon S'*S'\to P$ such that
	\begin{equation*}
		\begin{aligned}
			\phi c&=\tau^{S}b & \phi'c&=\tau^{S}a \\
			\pi_1 c&=s'd_{1}^{S'}       & \pi_0 c&=s'd_{0}^{S'}
		\end{aligned}
	\end{equation*}
	Then using this morphism $c$ we have the following commutative diagrams involving weak pullbacks with the corresponding induced factorizations.
	
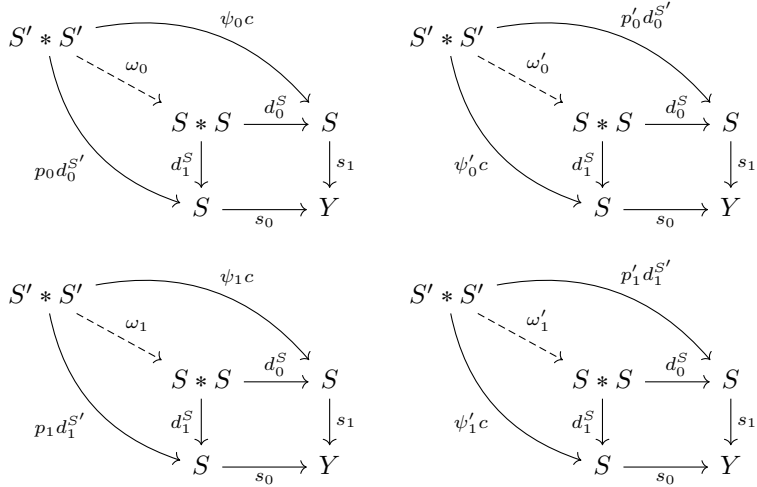
\begin{center}
		\begin{tikzcd}
			S'*S'\ar[dr,dashed,"\omega_0"]\ar[ddr,"{p_{0}d_{0}^{S'}}"',bend right]\ar[drr,"\psi_{0}c",bend left] & & \\
			& S*S\ar[r,"d_{0}^{S}"]\ar[d,"d_{1}^{S}"'] & S\ar[d,"s_1"] \\
			& S\ar[r,"s_0"'] & Y
		\end{tikzcd}
		\quad
		\begin{tikzcd}
			S'*S'\ar[dr,dashed,"{\omega_{0}'}"]\ar[ddr,"{\psi_{0}'c}"',bend right]\ar[drr,"p_{0}'d_{0}^{S'}",bend left] & & \\
			& S*S\ar[r,"d_{0}^{S}"]\ar[d,"d_{1}^{S}"'] & S\ar[d,"s_1"] \\
			& S\ar[r,"s_0"'] & Y
		\end{tikzcd}
	\end{center}
	\begin{center}
		\begin{tikzcd}
			S'*S'\ar[dr,dashed,"\omega_1"]\ar[ddr,"{p_{1}d_{1}^{S'}}"',bend right]\ar[drr,"\psi_{1}c",bend left] & & \\
			& S*S\ar[r,"d_{0}^{S}"]\ar[d,"d_{1}^{S}"'] & S\ar[d,"s_1"] \\
			& S\ar[r,"s_0"'] & Y
		\end{tikzcd}
		\quad
		\begin{tikzcd}
			S'*S'\ar[dr,dashed,"{\omega_{1}'}"]\ar[ddr,"{\psi_{1}'c}"',bend right]\ar[drr,"p_{1}'d_{1}^{S'}",bend left] & & \\
			& S*S\ar[r,"d_{0}^{S}"]\ar[d,"d_{1}^{S}"'] & S\ar[d,"s_1"] \\
			& S\ar[r,"s_0"'] & Y
		\end{tikzcd}
	\end{center}
	Now we have the following cone over the diagram defining $S'$:
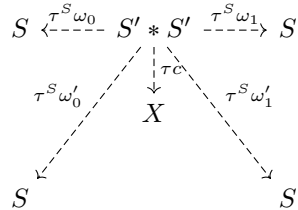
	\begin{center}
		\begin{tikzcd}
			S & S'*S'\ar[l,dashed,"\tau^{S}\omega_0"']\ar[ddl,dashed,"{\tau^{S}\omega_0'}"']\ar[d,dashed,"\tau c"]\ar[ddr,dashed,"{\tau^{S}\omega_1'}"]\ar[r,dashed,"\tau^{S}\omega_1"] & S   \\
			& X &  \\
			S & & S 
		\end{tikzcd}
	\end{center}
	Hence, we obtain a morphism $\tau^{S'}\colon S'*S'\to S'$ such that
	\begin{equation*}
		\begin{aligned}
			p_{0}\tau^{S'}&=\tau^{S}\omega_0 & p_{0}'\tau^{S'}&=\tau^{S}\omega_0' \\
			p_{1}\tau^{S'}&=\tau^{S}\omega_1       & p_{1}'\tau^{S'}&=\tau^{S}\omega_1'
		\end{aligned}
	\end{equation*}
	Then we also have:
	\begin{displaymath}
		s_{0}'\tau^{S'}=s_{0}p_{0}'\tau^{S'}=s_{0}\tau^{S}\omega_{0}'=s_{0}d_{0}^{S}\omega_{0}'=s_{0}p_{0}'d_{0}^{S'}=s_{0}'d_{0}^{S'}
	\end{displaymath}
	\begin{displaymath}
		s_{1}'\tau^{S'}=s_{1}p_{1}\tau^{S'}=s_{1}\tau^{S}\omega_{1}=s_{1}d_{1}^{S}\omega_{1}=s_{1}p_{1}d_{1}^{S'}=s_{1}'d_{1}^{S'}
	\end{displaymath}
	So, finally, $\tau^{S'}$ is the sought for transitivity of $\begin{tikzcd}[cramped] S'\ar[r,shift left=0.75ex,"{s_0'}"]\ar[r,shift right=0.75ex,"{s_1'}"'] & Y\end{tikzcd}$.
	\vspace{0.5cm}
	
	Order-reflexivity is not nearly as lengthy to check: consider any $y_{0},y_{1}\colon A\to Y\in\ca{C}$ with $y_0\leq y_1$. By order-reflexivity of $S$, the pair $(y_0,y_1)$ factors through $(s_0,s_1)$ via some $v\colon A\to S$. Then the following is easily seen to be a cone over the diagram defining $S'$:
	\begin{center}
		\begin{tikzcd}
			S & A\ar[l,dashed,"\Sigma_{0}v"']\ar[ddl,dashed,"{\Sigma_{0}'v}"']\ar[d,dashed,"uv"]\ar[ddr,dashed,"{\Sigma_{1}'v}"]\ar[r,dashed,"\Sigma_{1}v"] & S   \\
			& X &  \\
			S & & S 
		\end{tikzcd}
	\end{center}
	Thus, there is a $v'\colon A\to S$ such that $p_{0}v'=\Sigma_{0}v$, $s'v'=uv$, $p_{0}'v'=\Sigma_{0}'v$, $p_{1}'v'=\Sigma_{1}'v$ and $p_{1}v'=\Sigma_{1}v$. We then also have
	\begin{displaymath}
		s_{0}'v'=s_{0}p_{0}'v'=s_{0}\Sigma_{0}'v=s_{0}v=y_0
	\end{displaymath}
	\begin{displaymath}
		s_{1}'v'=s_{1}p_{1}v'=s_{1}\Sigma_{1}v=s_{1}v=y_1
	\end{displaymath}
	\vspace{0.5cm}
	
	Finally, now that we have established that $(Y,S')$ is a legitimate object of $\ca{C}_{\rm{ex}}$ and that there is a morphism $[1_Y]\colon (Y,S)\to(Y,S')$, we show that $\begin{tikzcd}[cramped] (X,R)\ar[r,shift left=0.75ex,"{[h_0]}"]\ar[r,shift right=0.75ex,"{[h_1]}"'] & (Y,S) \end{tikzcd}$ is the kernel congruence of this morphism. So consider any morphisms $\begin{tikzcd}[cramped] (Z,T)\ar[r,shift left=0.75ex,"{[f_0]}"]\ar[r,shift right=0.75ex,"{[f_1]}"'] & (Y,S) \end{tikzcd}$ in $\ca{C}_{\rm{ex}}$ with $[1_Y]\circ[f_0]\leq[1_Y]\circ[f_1]$. So there is  $\Sigma\colon Z\to S'\in\ca{C}$ with $s_{0}'\Sigma=f_0$ and $s_{1}'\Sigma=f_1$. Set $f\coloneqq s'\Sigma\colon Z\to X$. We claim that $f$ defines a morphism $(Z,T)\to(X,R)$ which is then clearly the desired factorization in $\ca{C}_{\rm{ex}}$.
	
	Note that $ft_0$ and $ft_1$ trivially define morphisms $(T,I_T)\to (X,R)$, where $I_T$ is a weak comma square of $(1_T,1_T)$ in $\ca{C}$. If we show that $[ft_0]\leq[ft_1]$, then this will entail that $f$ defines a morphism as wanted, since the two statements say exactly the same thing. In turn, the latter inequality is equivalent to the pair of inequalities $[h_0]\circ[ft_0]\leq[h_0]\circ[ft_1]$ and $[h_1]\circ[ft_0]\leq[h_1]\circ[ft_1]$, since $([h_0],[h_1])$ is a jointly order-monic pair in $\ca{C}_{\rm{ex}}$. To prove that these two inequalities indeed hold, consider the following two diagrams in succession.
	\begin{center}
		\begin{tikzcd}
			T\ar[dr,dashed,"a"]\ar[ddr,"\bar{f}_{0}"',bend right]\ar[drr,"p_{0}\Sigma t_0",bend left] & & \\
			& S*S\ar[r,"d_{0}^{S}"]\ar[d,"d_{1}^{S}"'] & S\ar[d,"s_1"] \\
			& S\ar[r,"s_0"'] & Y
		\end{tikzcd}
		\quad
		\begin{tikzcd}
			T\ar[dr,dashed,"b"]\ar[ddr,"{p_{0}'\Sigma t_1}"',bend right]\ar[drr,"\tau^{S}a",bend left] & & \\
			& S*S\ar[r,"d_{0}^{S}"]\ar[d,"d_{1}^{S}"'] & S\ar[d,"s_1"] \\
			& S\ar[r,"s_0"'] & Y
		\end{tikzcd}
	\end{center}
	Now the morphism $\tau^{S}b\colon T\to S$ satisfies
	\begin{displaymath}
		s_{0}(\tau^{S}b)=s_{0}d_{0}^{S}b=s_{0}\tau^{S}a=s_{0}d_{0}^{S}a=s_{0}p_{0}\Sigma t_0=h_{0}s'\Sigma t_0=h_{0}ft_0
	\end{displaymath}
	\begin{displaymath}
		s_{1}(\tau^{S}b)=s_{1}d_{1}^{S}b=s_{1}p_{0}'\Sigma t_1=h_{0}s'\Sigma t_1=h_{0}ft_1
	\end{displaymath}
	which proves that $[h_0]\circ[ft_0]\leq[h_0]\circ[ft_1]$. The second inequality follows in an entirely similar fashion.
	
	Thus, we indeed have a morphism $[f]\colon (Z,T)\to (X,R)$ in $\ca{C}_{\rm{ex}}$. We now observe that 
	\begin{align*}
		s_{0}(p_{0}\Sigma)=h_{0}s'\Sigma=h_{0}f &, \hspace{0.3cm} s_{1}(p_{0}\Sigma)=s_{0}'\Sigma=f_0 \\
		s_{0}(p_{0}'\Sigma)=s_{0}'\Sigma=f_0 &, \hspace{0.3cm} s_{1}(p_{0}'\Sigma)=h_{0}s'\Sigma=h_0 f
	\end{align*}
	which together show that $[h_0]\circ[f]=[f_0]$. Similarly, we also have $[h_1]\circ[f]=[f_1]$.
\end{proof}

Collecting together all of our results on $\ca{C}_{ex}$ so far we have the following.

\begin{thm}
	The category $\ca{C}_{\rm{ex}}$ is exact.
\end{thm}

\begin{rmk}
	The astute reader might have noticed that in all constructions and proofs above we have almost exclusively used \emph{conical} weak finite limits. The only exception has been the need for weak comma objects for pairs of identity morphisms. However, it is not too hard to check that the existence of these particular types of weak limits already implies the existence of all weak finite limits. The same is of course true for their strong counterparts.
\end{rmk}

Next, we want to define a functor $\Gamma\colon\ca{C}\to\ca{C}_{\rm{ex}}$ which will be an embedding of $\ca{C}$ into $\ca{C}_{\rm{ex}}$. In order to do this, we pick for every object $X\in\ca{C}$ a weak comma square
\begin{center}
	\begin{tikzcd}
		I_{X}\ar[d,"i_{0}^{X}"']\ar[r,"i_{1}^{X}"]\ar[dr,phantom,"\leq"] & X\ar[d,equal] \\
		X\ar[r,equal] & X
	\end{tikzcd}
\end{center}
and set $\Gamma X\coloneqq(X,I_X)$. Now for any morphism $f\colon X\to Y\in\ca{C}$ we have that $i_{0}^{X}\leq i_{1}^{X}\implies fi_{0}^{X}\leq fi_{1}^{X}$, hence $(fi_{0}^{X}, fi_{1}^{X})$ factors through $I_{Y}$. Thus, we have a morphism $\Gamma f\coloneqq [f]\colon (X,I_X)\to(Y,I_Y)$ in $\ca{C}_{\rm{ex}}$. It is clear that this assignment preserves composition and identities. In addition, simply by definition of $I_Y$, we have for any $f,g\colon X\to Y\in\ca{C}$ that $f\leq g\iff\Gamma f\leq\Gamma g$. In other words, $\Gamma$ is a fully order-faithful functor $\ca{C}\to\ca{C}_{\rm{ex}}$.

\begin{rmk}
	Observe that $\Gamma$ is essentially independent of the choice of weak comma square for every $X\in\ca{C}$. Indeed, if $I_{X}'$ is another weak comma square of $(1_X,1_X)$ in $\ca{C}$, then $I_X$ and $I_{X}'$ factor through each other. This implies that $1_X$ defines morphisms $(X,I_X)\to(X,I_{X}')$ and $(X,I_{X}')\to(X,I_X)$ in $\ca{C}_{\rm{ex}}$, which are then clearly mutually inverse isomorphisms. These are also compatible with the definition of $\Gamma$ on morphisms. 
\end{rmk}

We end this section with a couple of observations concerning the manner in which $\ca{C}$ is embedded in $\ca{C}_{\rm{ex}}$. These are again precise enriched analogues of familiar facts about the ordinary exact completion \cite{Carboni-Vitale}. 

\begin{prop}\label{C generates C_ex via coinserters}
	Every morphism $[f]\colon(X,R)\to(Y,S)\in\ca{C}_{\rm{ex}}$ occurs in a commutative diagram of the following form
	\begin{center}
		\begin{tikzcd}
			\Gamma R\ar[r,shift left=0.75ex,"\Gamma r_0"]\ar[r,shift right=0.75ex,"\Gamma r_1"']\ar[d,"\Gamma\bar{f}"'] & \Gamma X\ar[d,"\Gamma f"]\ar[r,two heads,"{[1_X]}"] & (X,R)\ar[d,"{[f]}"] \\
			\Gamma S\ar[r,shift left=0.75ex,"\Gamma s_0"]\ar[r,shift right=0.75ex,"\Gamma s_1"'] & \Gamma Y\ar[r,two heads,"{[1_Y]}"] & (Y,S)
		\end{tikzcd}
	\end{center}
	where the rows are coinserters.
\end{prop}
\begin{proof}
	The diagram is clearly commutative, so it suffices to show that	the row
	\begin{tikzcd}[cramped]
		\Gamma R\ar[r,shift left=0.75ex,"\Gamma r_0"]\ar[r,shift right=0.75ex,"\Gamma r_1"'] & \Gamma X\ar[r,two heads,"{[1_X]}"] & (X,R)
	\end{tikzcd}
	is a coinserter in $\ca{C}_{\rm{ex}}$. It is a tautology that $[1_X]\circ\Gamma r_0\leq[1_X]\circ\Gamma r_1$. Now let $[f]\colon\Gamma X\to(Y,S)\in\ca{C}_{\rm{ex}}$ be such that $[f]\circ\Gamma r_0\leq[f]\circ\Gamma r_1$. This means that $(fr_0,fr_1)$ factors through $(s_0,s_1)$, which is the same as the statement that $f$ also defines a morphism $(X,R)\to(Y,S)$.
	\begin{center}
		\begin{tikzcd}
			\Gamma R\ar[r,shift left=0.75ex,"\Gamma r_0"]\ar[r,shift right=0.75ex,"\Gamma r_1"'] & \Gamma X\ar[r,two heads,"{[1_X]}"]\ar[dr,"{[f]}"'] &  (X,R)\ar[d,dashed,"{[f]}"] \\
			& &  (Y,S)  
		\end{tikzcd}
	\end{center}
	Finally, it is clear that $[1_X]$ is order-epimorphic.
\end{proof}

\begin{prop}
	For every $A\in\ca{C}$, the object $\Gamma A$ is projective in $\ca{C}_{\rm{ex}}$.
\end{prop}
\begin{proof}
	Recall that, up to isomorphism, every $\nc{so}$-morphism in $\ca{C}_{\rm{ex}}$ can be represented by an identity morphism of $\ca{C}$. So consider such a morphism $[1_X]\colon(X,R)\to(X,S)$ and an arbitrary morphism $[f]\colon\Gamma Z\to(X,S)$ in $\ca{C}_{\rm{ex}}$. Then $f$ also defines a morphism  $\Gamma Z\to (X,R)$, simply by order-reflexivity of $R$, yielding the desired factorization.
\end{proof}

In particular, we have that $\ca{C}$ is a projective cover of $\ca{C}_{\rm{ex}}$.

\section{The Universal Property and Characterization of the Completion}

In the case of the exact completion $\ca{C}_{\rm{ex}}$ of an ordinary category $\ca{C}$, a fact which at first appeared curious is that its universal property does not involve the type of functors that one might expect.

To state a prospective universal property, one has in any case to consider functors $\ca{C}\to\ca{E}$ from the merely weakly lex category $\ca{C}$ into exact categories $\ca{E}$. It is natural to initially expect that the appropriate functors should be the weakly lex ones, i.e. the ones preserving weak finite limits, and that any such functor will uniquely extend to an exact functor $\ca{C}_{\rm{ex}}\to\ca{E}$. In other words, this is to say that the construction of the exact completion should be the left biadjoint to the obvious forgetful 2-functor $\nc{EX}\to\nc{WLEX}$. However, as observed in \cite{Carboni-Vitale}, this is simply not the case. Furthermore, the authors show that this issue cannot be avoided by choosing a different class of functors. It is simply impossible to turn the construction $\ca{C}_{\rm{ex}}$ into a left biadjoint.

Instead, the universal property of $\ca{C}_{\rm{ex}}$ is of a different flavor. The identification of the functors involved is a fundamental contribution of Enrico Vitale. He has termed them \emph{left covering} and they form a central part (and provide the title) of his 1994 PhD thesis \cite{VitaleThesis}. It is this notion which we now simply adapt to the $\nc{Pos}$-enriched context.

\begin{defi}\label{defi: left covering}
	Consider a functor $F\colon\ca{C}\to\ca{E}$ where $\ca{C}$ is weakly lex and $\ca{E}$ is regular. We say that $F$ is \emph{left covering} if, given any finite weight $W\colon\ca{I}\to\nc{Pos}$ and diagram $D\colon\ca{I}\to\ca{C}$, for every weak limit $\{W,D\}_{w}$ in $\ca{C}$ the canonical factorization $F\{W,D\}_{w}\to\{W,FD\}$ is an effective epimorphism in $\ca{E}$.
\end{defi}

We will not go into an extensive analysis of the properties that such functors enjoy, but rather restrict ourselves to those that are strictly necessary for proving the universal property of the exact completion. Perhaps the most fundamental is the one contained in the following, which should be compared to Theorem 26 of \cite{Carboni-Vitale}.

\begin{thm}\label{left covering functors and congruences}
	Let $F\colon\ca{C}\to\ca{E}$ be a left covering functor, where $\ca{C}$ is weakly lex and $\ca{E}$ is regular. If $\begin{tikzcd}[cramped] R\ar[r,shift left=0.75ex,"r_0"]\ar[r,shift right=0.75ex,"r_1"'] & X\end{tikzcd}$ is a pseudocongruence in $\ca{C}$ and $\begin{tikzcd}[cramped]FR\ar[r,two heads,"p"] & S\ar[r,tail,"{\langle s_0,s_1\rangle}"] & FX \times FX\end{tikzcd}$ is the $(\nc{so},\nc{ff})$ factorization of $\langle Fr_0,Fr_1\rangle$ in $\ca{E}$, then $\begin{tikzcd}[cramped] S\ar[r,shift left=0.75ex,"s_0"]\ar[r,shift right=0.75ex,"s_1"'] & FX\end{tikzcd}$ is a congruence in $\ca{E}$.
\end{thm}
\begin{proof}
	Consider the following weak comma in $\ca{C}$ and comma in $\ca{E}$ respectively.
	\begin{center}
		\begin{tikzcd}
			I_{X}\ar[d,"c_0"']\ar[r,"c_1"]\ar[dr,phantom,"\leq"] & X\ar[d,equal] \\
			X\ar[r,equal] & X
		\end{tikzcd}
		\quad 
		\begin{tikzcd}
			I_{FX}\ar[d,"i_0"']\ar[r,"i_1"]\ar[dr,phantom,"\leq"] & FX\ar[d,equal] \\
			FX\ar[r,equal] & FX
		\end{tikzcd}
	\end{center}
	By order-reflexivity of $R$, there is a $u\colon I_X\colon\to R\in\ca{C}$ such that $r_{0}u=c_0$ and $r_{1}u=c_1$. Since $F$ is left covering, the comparison morphism $q\colon FX\twoheadrightarrow I_{FX}$ in an $\nc{so}$-morphism in $\ca{E}$. We then have the following commutative diagram:
	\begin{center}
		\begin{tikzcd}
			FI_{X}\ar[r,two heads,"q"]\ar[d,"Fu"'] & I_{FX}\ar[r,tail,"{\langle i_0,i_1\rangle}"]\ar[d,dashed] & FX\times FX\ar[d,equal] \\
			FR\ar[r,two heads,"p"'] & S\ar[r,tail,"{\langle s_0,s_1\rangle}"'] & FX\times FX
		\end{tikzcd}
	\end{center}
	The arrow $I_{FX}\to S$ is induced by functoriality of $(\nc{so},\nc{ff})$ factorizations and exhibits order-reflexivity of $S$.
	
	For transitivity, we observe first that the induced arrow $\epsilon\colon FR*FR\twoheadrightarrow S*S$ is an $\nc{so}$-morphism. This can be seen by forming the following diagram of pullback squares and using the stability of $\nc{so}$-morphisms under pullback in the regular category $\ca{E}$:
	\begin{center}
		\begin{tikzcd}
			FR*FR\ar[r,two heads]\ar[d,two heads] & P\ar[r]\ar[d,two heads] & FR\ar[d,two heads,"p"] \\
			P'\ar[r,two heads]\ar[d] & S*S\ar[r]\ar[d] & S\ar[d,tail,"{\langle s_0,s_1\rangle}"] \\
			FR\ar[r,two heads,"p"'] & S\ar[r,tail,"{\langle s_0,s_1\rangle}"'] & FX\times FX
		\end{tikzcd}
	\end{center}
	In addition, the canonical $q\colon F(R*R)\twoheadrightarrow FR*FR$ is an $\nc{so}$-morphism because $F$ is left covering. Using the transitivity of $R$ we finally have the following commutative diagram:
	\begin{center}
		\begin{tikzcd}
			F(R*R)\ar[r,two heads,"\epsilon q"]\ar[d,"p\circ F\tau^{R}"'] & S*S\ar[d,"{\langle s_{0}d_{0}^{S},s_{1}d_{1}^{S}\rangle}"]\ar[dl,dashed,"\tau^{S}"'] \\
			S\ar[r,tail,"{\langle s_0,s_1\rangle}"'] & FX\times FX
		\end{tikzcd}
	\end{center}
	
	The arrow $\epsilon q$ is an $\nc{so}$-morphism because both $\epsilon,q$ are such. Since $\langle s_0,s_1\rangle$ is order-monic, we get a $\tau^{S}\colon S*S\to S$ making both triangles commute. This gives transitivity of $S$.
\end{proof}

It will also be important to have a simpler way of checking that a functor is left covering. The next proposition reduces this check to two specific types of finite limits. However, we begin with a lemma which will be used multiple times here and later on.

\begin{lem}\label{useful lemma}
	Suppose that $\ca{E}$ is a regular category.
\begin{enumerate}
	\item Consider the following commutative diagram in $\ca{E}$, where the rows are inserters, $m$ is an order-mono and $p$ is an effective epi. 
	\begin{center}
		\begin{tikzcd}
			I\ar[d,dashed,"q"']\ar[r,tail,"i"] & X\ar[d,two heads,"p"']\ar[r,shift left=0.75ex,"f_0"]\ar[r,shift right=0.75ex,"f_1"'] & Y\ar[d,tail,"m"] \\
			I'\ar[r,tail,"{i'}"] & X'\ar[r,shift left=0.75ex,"{f'_0}"]\ar[r,shift right=0.75ex,"{f'_1}"'] & Y'
		\end{tikzcd}
	\end{center}
	Then the induced morphism $q\colon I\to I'$ is also an effective epi.
	\item Consider the following finite family of commutative squares in $\ca{E}$
	\begin{displaymath}
		\begin{tikzcd}
			A_k\ar[r,"a_k"]\ar[d,two heads,"f_k"'] & A\ar[d,tail,"f"] \\
			B_k\ar[r,"b_k"'] & B
		\end{tikzcd}
	\end{displaymath}
	where the $f_k$ are effective epis and $f$ is an order-mono. Then the induced morphism from the $n$-ary pullback of the $a_k$ to that of the $b_k$ is an effective epi.
\end{enumerate}
\end{lem}
\begin{proof}
(1)	We will show that the left-hand square is a pullback. So let $g\colon Z\to I$ and $h\colon Z\to X$ be such that $i'g=ph$. Then
	$$mf_{0}h=f'_{0}ph=f'_{0}i'g\leq f'_{1}i'g=f'_{1}ph=mf_{1}h$$
	and because $m$ is an order-mono we obtain $f_{0}h\leq f_{1}h$. Thus, there is a unique $u\colon Z\to I$ such that $iu=h$. Then we also have that $qu=g$, since $i'qu=piu=ph=i'g$ and $i'$ is order-monic. Finally, the pair $(q,i)$ is jointly order-monic because $i$ on its own is already so.  
	
	We also observe here that the same (almost) exact proof works if the two rows are equalizers instead of inserters. This then generalizes straightforwardly to $n$-ary equalizers for any $n\geq 2$. We shall need this form of the result below.
	
(2) Form the following commutative diagram, where the rows are $n$-ary equalizers.
\begin{displaymath}
	\begin{tikzcd}[sep=3.5em]
		P\ar[r,tail]\ar[d,"r"'] & A_1\times...\times A_n\ar[r,shift left=1.5ex,"a_1\pi_1"]\ar[r,shift right=1.5ex,"a_n\pi_n"',"\vdots"]\ar[d,"f_1\times...\times f_n"'] & A\ar[d,tail,"f"] \\
		Q\ar[r,tail] & B_1\times...\times B_n\ar[r,shift left=1.5ex,"b_1\pi_1"]\ar[r,shift right=1.5ex,"\vdots","b_n\pi_n"'] & B
	\end{tikzcd}
\end{displaymath}
Note that $f_1\times...\times f_n$ is an effective epi, because each $f_k$ is such and we are in a regular category. Then by (the $n$-ary equalizer version of) part (1) we have that $r\colon P\to Q$ is an effective epi. But the latter is precisely the comparison morphism between the $n$-ary pullbacks.
\end{proof}

\begin{prop}\label{left covering wrt prod and ins}
	Let $F\colon\ca{C}\to\ca{E}$ be functor from the weakly lex category $\ca{C}$ to the regular category $\ca{E}$. If $F$ is left covering with respect to terminal object, finite products and inserters, then it is left covering.
\end{prop} 
\begin{proof}
This proof will proceed in a sequence of steps, each verifying that $F$ is left covering with respect to a particular type of limit, before the final step dealing with the general case. 

\underline{$F$ is left covering with respect to $n$-ary products:} It suffices to deal with the case $n=3$. So consider objects $A,B,C\in\ca{C}$ and take successive weak binary products $\begin{tikzcd}
	A & P\ar[l,"\pi_A"']\ar[r,"\pi_B"] & B
\end{tikzcd}$ and
$\begin{tikzcd}
	P& Q\ar[l,"\pi_P"']\ar[r,"\pi_C"] & C
\end{tikzcd}$ in $\ca{C}$. Then $Q$ is a weak ternary product and the comparison morphism to the product in $\ca{E}$ can be written as the composition $\begin{tikzcd}
FQ\ar[r,two heads,"q"] & FP\times FC\ar[r,two heads,"p\times 1"] & (FA\times FB)\times FC\cong FA\times FB\times FC
\end{tikzcd}$, where $p,q$ are the corresponding comparisons for the binary products. Since we are in a regular category, $p$ being an effective epi implies the same property for $p\times 1$ and then for the composition $(p\times 1)q$.

The general case can similarly be obtained by induction, constructing for each $n$ the $(n+1)$-ary product via the $n$-ary one and binary products.

\underline{$F$ is left covering with respect to equalizers:} Consider a pair of morphisms $\begin{tikzcd}
	X\ar[r,shift left=1ex,"f"]\ar[r,shift right=1ex,"g"'] & Y
\end{tikzcd}$ in $\ca{C}$. Form weak inserters 
$\begin{tikzcd}
E\ar[r,"e"] & X\ar[r,shift left=1ex,"f"]\ar[r,shift right=1ex,"g"',"\vleq"] & Y
\end{tikzcd}$
and
$\begin{tikzcd}
	E'\ar[r,"{e'}"] & E\ar[r,shift left=1ex,"ge"]\ar[r,shift right=1ex,"fe"',"\vleq"] & Y
\end{tikzcd}$, so that then 
$\begin{tikzcd}
E'\ar[r,"{ee'}"] & X\ar[r,shift left=1ex,"f"]\ar[r,shift right=1ex,"g"'] & Y
\end{tikzcd}$
is a weak equalizer.  Similarly, we form strong inserters 
$\begin{tikzcd}
	I\ar[r,tail,"i"] & FX\ar[r,shift left=1ex,"Ff"]\ar[r,shift right=1ex,"Fg"',"\vleq"] & FY
\end{tikzcd}$
and
$\begin{tikzcd}
	I'\ar[r,tail,"{i'}"] & I\ar[r,shift left=1ex,"Fg\circ i"]\ar[r,shift right=1ex,"Ff\circ i"',"\vleq"] & FY
\end{tikzcd}$,
so that 
$\begin{tikzcd}
	I'\ar[r,tail,"{ii'}"] & FX\ar[r,shift left=1ex,"Ff"]\ar[r,shift right=1ex,"Fg"'] & FY
\end{tikzcd}$
is an equalizer in $\ca{E}$. We also form the inserter
$\begin{tikzcd}
	I''\ar[r,tail,"{i''}"] & FE\ar[r,shift left=1ex,"F(ge)"]\ar[r,shift right=1ex,"F(fe)"',"\vleq"] & FY
\end{tikzcd}$.

Observe that, since $F$ is left covering with respect to inserters, the morphism $q\colon FE'\to I''$ such that $i''q=Fe'$ is an effective epi. Now we have the following commutative diagram
\begin{displaymath}
\begin{tikzcd}
	I''\ar[r,tail,"{i''}"]\ar[d,"r"'] & FE\ar[d,two heads,"p"']\ar[r,shift left=1ex,"F(ge)"]\ar[r,shift right=1ex,"F(fe)"',"\vleq"] & FY\ar[d,equal] \\
	I'\ar[r,tail,"{i'}"] & I\ar[r,shift left=1ex,"Fg\circ i"]\ar[r,shift right=1ex,"Ff\circ i"',"\vleq"] & FY
\end{tikzcd}
\end{displaymath}
where $p$ is the comparison morphism with $ip=Fe$, which is thus an effective epi because $F$ is left covering with respect to inserters. Then by \ref{useful lemma} we deduce that $r$ is also an effective epi. Finally, $rq\colon FE'\to I'$ is now an effective epi and it is indeed the comparison morphism to the equalizer $I'$, since $ii'q=ipi''q=FeFe'=F(ee')$.

\underline{$F$ is left covering with respect to $n$-ary equalizers:} As with products, it suffices to deal with the ternary case, as the general one then follows in the same way by induction. So consider a triple of arrows 
$\begin{tikzcd}
	X\ar[r,shift left=1.5ex,"f"]\ar[r,shift right=1.5ex,"h"']\ar[r,"g"] & Y
\end{tikzcd}$ in $\ca{C}$.

We form weak equalizers
$\begin{tikzcd}
	E\ar[r,"e"] & X\ar[r,shift left=1ex,"f"]\ar[r,shift right=1ex,"g"'] & Y
\end{tikzcd}$
and
$\begin{tikzcd}
	E'\ar[r,"{e'}"] & E\ar[r,shift left=1ex,"ge"]\ar[r,shift right=1ex,"he"'] & Y
\end{tikzcd}$
in $\ca{C}$, so that the diagram
$\begin{tikzcd}
	E'\ar[r,"{ee'}"] & 	X\ar[r,shift left=1.5ex,"f"]\ar[r,shift right=1.5ex,"h"']\ar[r,"g"] & Y
\end{tikzcd}$
is a weak ternary equalizer. Then in $\ca{E}$ we form the equalizers
$\begin{tikzcd}
	I\ar[r,tail,"i"] & FX\ar[r,shift left=1ex,"Ff"]\ar[r,shift right=1ex,"Fg"'] & FY
\end{tikzcd}$
and
$\begin{tikzcd}
	I'\ar[r,tail,"{i'}"] & I\ar[r,shift left=1ex,"Fg\circ i"]\ar[r,shift right=1ex,"Fh\circ i"'] & FY
\end{tikzcd}$,
so that
$\begin{tikzcd}
	I'\ar[r,tail,"{ii'}"] & FX\ar[r,shift left=1.5ex,"Ff"]\ar[r,shift right=1.5ex,"Fh"']\ar[r,"Fg"] & FY
\end{tikzcd}$
is a ternary equalizer. Also form the equalizer
$\begin{tikzcd}
	I''\ar[r,tail,"{i''}"] & FE\ar[r,shift left=1ex,"F(ge)"]\ar[r,shift right=1ex,"F(he)"'] & FY
\end{tikzcd}$.

By the previous step we know that $F$ is left covering with respect to (binary) equalizers, hence the comparison maps $q\colon FE\to I$, such that $iq=FE$, and $q''\colon FE'\to I''$, such that $i''q''=Fe'$, are both effective epis. Now we have the following commutative diagram
\begin{displaymath}
\begin{tikzcd}
	I''\ar[r,tail,"{i''}"]\ar[d,"r"'] & FE\ar[d,two heads,"q"']\ar[r,shift left=1ex,"F(ge)"]\ar[r,shift right=1ex,"F(he)"'] & FY\ar[d,equal] \\
	I'\ar[r,tail,"{i'}"] & I\ar[r,shift left=1ex,"Fg\circ i"]\ar[r,shift right=1ex,"Fh\circ i"'] & FY
\end{tikzcd}
\end{displaymath}
By the equalizer version of \ref{useful lemma} it follows that $r$ is also an effective epi and hence so is $rq''\colon FE'\twoheadrightarrow I'$. But the latter is precisely the comparison map to the ternary equalizer.

\underline{$F$ is left covering with respect to $n$-ary pullbacks:} Consider a finite family of arrows $f_k\colon X_k\to Y$, $k=1,...,n$, in $\ca{C}$. Take a weak $n$-ary product diagram $\pi_k\colon P\to X_k$ and then a weak $n$-ary equalizer
\begin{tikzcd}
	E\ar[r,"e"] & P\ar[r,shift left=1.5ex,"f_1\pi_1"]\ar[r,shift right=1.5ex,"f_n\pi_n"',"\vdots"] & Y
\end{tikzcd},
so that
\begin{displaymath}
\begin{tikzcd}
	& & E\ar[dll,"\pi_1e"']\ar[dl,"\pi_2e"]\ar[drr,"\pi_ne"] & & \\
	X_1\ar[drr,"f_1"'] & X_2\ar[dr,"f_2"] & & \hdots & X_n\ar[dll,"f_n"] \\
	& & Y & &
\end{tikzcd}
\end{displaymath}
is a weak $n$-ary pullback. At the same time, consider in $\ca{E}$ the $n$-ary equalizer
\begin{tikzcd}
	I\ar[r,tail,"i"] & \prod\limits_{k=1}^{n}FX_k\ar[r,shift left=1.5ex,"Ff_1\pi_{FX_1}"]\ar[r,shift right=1.5ex,"Ff_n\pi_{FX_n}"',"\vdots"] & FY
\end{tikzcd},
so that the following is an $n$-ary pullback
\begin{displaymath}
\begin{tikzcd}
& & I\ar[dll,"\pi_{FX_1}i"']\ar[dl,"\pi_{FX_2}i"]\ar[drr,"\pi_{FX_n}i"] & & \\
FX_1\ar[drr,"Ff_1"'] & FX_2\ar[dr,"Ff_2"] & & \hdots & FX_n\ar[dll,"Ff_n"] \\
& & FY & &
\end{tikzcd}
\end{displaymath}
We also form the $n$-ary equalizer
\begin{tikzcd}
	I'\ar[r,tail,"{i'}"] & FP\ar[r,shift left=1.5ex,"F(f_1\pi_1)"]\ar[r,shift right=1.5ex,"F(f_n\pi_n)"',"\vdots"] & FY
\end{tikzcd}
By the previous steps in the proof we know that the unique $q\colon FE\to I'$ such that $i'q=Fe$ and the comparison $p\colon FP\to\prod\limits_{k=1}^{n}FX_k$ are both effective epis. Then we have the following commutative diagram
\begin{displaymath}
\begin{tikzcd}
	I'\ar[r,tail,"{i'}"]\ar[d,"r"] & FP\ar[r,shift left=1.5ex,"F(f_1\pi_1)"]\ar[r,shift right=1.5ex,"F(f_n\pi_n)"',"\vdots"]\ar[d,two heads,"p"'] & FY\ar[d,equal] \\
	I\ar[r,tail,"i"] & \prod\limits_{k=1}^{n}FX_k\ar[r,shift left=1.5ex,"Ff_1\pi_{FX_1}"]\ar[r,shift right=1.5ex,"Ff_n\pi_{FX_n}"',"\vdots"] & FY
\end{tikzcd}
\end{displaymath}
It follows once more by \ref{useful lemma} that $r\colon I\to I'$ is an effective epi. Hence, so is $rq\colon FE\to I$, which is precisely the comparison map to the $n$-ary pullback.

\underline{$F$ is left covering:} Since this proof is already lengthy, to avoid over-complicating the issue we will only present the proof for a particular type of non-conical finite weighted limit. This (along with the previous steps of the proof) should already be indicative of the general situation, but is also technically the only type for which we will need the result in what follows. See also the construction of weak finite limits by weak finite products and weak inserters given in \ref{weakly lex from weak prod and weak ins}.

We consider the weak \emph{joint kernel} in $\ca{C}$ of a family of arrows $f_k\colon X\to Y_k$, for $k=1,...,n$. This is by definition a pair $\begin{tikzcd}
	L\ar[r,shift left=0.75ex,"l_0"]\ar[r,shift right=0.75ex,"l_1"'] & X
\end{tikzcd}$ which is weakly universal with respect to the property that $f_kl_0\leq f_kl_1$ for all $k=1,...,n$. In other words, $L$ is a (weak) limit weighted by $W\colon\ca{I}\to\nc{Pos}$, where $\ca{I}$ is a collection of spans $\begin{tikzcd}
i_0\ar[dr] && i_1\ar[dl] \\
& j_k &
\end{tikzcd}$, one for each $k=1,...,n$, and $W$ acts as follows:
\begin{displaymath}
\begin{tikzcd}
	i_0\ar[dr] && i_1\ar[dl] \\
	& j_k &
\end{tikzcd}
\qquad
\mapsto
\qquad
\begin{tikzcd}
	\{0\}\ar[dr,hook] && \{1\}\ar[dl,hook'] \\
	&  \{0\leq 1\}  &
\end{tikzcd}
\end{displaymath}
We would like to show that $F$ is left covering with respect to this weighted limit.

We construct $L$ in steps. First, let $\begin{tikzcd}
	X & P\ar[l,"\pi_0"']\ar[r,"\pi_1"] & X
\end{tikzcd}$ be a weak binary product in $\ca{C}$. Then for each $k=1,...,n$ we take a weak inserter $\begin{tikzcd}
E_k\ar[r,"e_k"] & P\ar[r,shift left=1ex,"f_k\pi_0"]\ar[r,shift right=1ex,"f_k\pi_1"',"\vleq"] & Y_k
\end{tikzcd}$, and finally a weak $n$-ary pullback 
\begin{displaymath}
\begin{tikzcd}
	& & L\ar[dll,"m_1"']\ar[dl,"m_2"]\ar[drr,"m_n"] & & \\
	E_1\ar[drr,"e_1"'] & E_2\ar[dr,"e_2"] & & \hdots & E_n\ar[dll,"e_n"] \\
	& & P & &
\end{tikzcd}
\end{displaymath}
so that we can set $l_0=\pi_0e_km_k$ and $l_1=\pi_1e_km_k$.

At the same time, we perform the constructions of the corresponding strong limits in $\ca{E}$ as indicated below.
\begin{displaymath}
\begin{tikzcd}
	I_k\ar[r,tail,"i_k"] & FX\times FX\ar[r,shift left=1ex,"Ff_k\pi_0"]\ar[r,shift right=1ex,"Ff_k\pi_1"',"\vleq"] & FY_k
\end{tikzcd}
\qquad
\begin{tikzcd}
	& & \Lambda\ar[dll,"\lambda_1"']\ar[dl,"\lambda_2"]\ar[drr,"\lambda_n"] & & \\
	I_1\ar[drr,tail,"i_1"'] & I_2\ar[dr,tail,"i_2"] & & \hdots & I_n\ar[dll,tail,"i_n"] \\
	& & FX\times FX & &
\end{tikzcd}
\end{displaymath}
Define $i\colon\Lambda\to FX\times FX$ by $i\coloneqq i_k\lambda_k$ and note that $i$ is precisely the joint inserter of the finite family $\begin{tikzcd}
	FX\times FX\ar[r,shift left=1ex,"Ff_k\pi_0"]\ar[r,shift right=1ex,"Ff_k\pi_1"',"\vleq"] & FY_k
\end{tikzcd}$, $k=1,...,n$ of ordered pairs of morphisms. We also form the following $n$-ary pullback in $\ca{E}$.
\begin{displaymath}
\begin{tikzcd}
	& & \Lambda'\ar[dll,"{\lambda_{1}'}"']\ar[dl,"{\lambda_{2}'}"]\ar[drr,"{\lambda_{n}'}"] & & \\
	FE_1\ar[drr,"Fe_1"'] & FE_2\ar[dr,"Fe_2"] & & \hdots & FE_n\ar[dll,"Fe_n"] \\
	& & FP & &
\end{tikzcd}
\end{displaymath}
Then we already know that the unique $r\colon FE\to\Lambda'$ such that $\lambda_{k}'r=Fm_k$ for all $k=1,...,n$ is an effective epi. Now consider the following pullback in $\ca{E}$.
\begin{displaymath}
\begin{tikzcd}
	Q\ar[r,tail,"{i'}"]\ar[d,two heads,"{p'}"'] & FP\ar[d,two heads,"p"] \\
	\Lambda\ar[r,tail,"i"'] & FX\times FX
\end{tikzcd}
\end{displaymath}
Then $p'$ is also an effective epi. In addition, an easy diagram chase shows that $i'$ is the joint inserter of the family $\begin{tikzcd}
	FP\ar[r,shift left=1ex,"F(f_k\pi_0)"]\ar[r,shift right=1ex,"F(f_k\pi_1)"'] & FY_k
\end{tikzcd}$, $k=1,...,n$. But this means that $i'$ can equally be constructed first taking inserters $\begin{tikzcd}
I_{k}'\ar[r,tail,"{i_{k}'}"] & FP\ar[r,shift left=1ex,"F(f_k\pi_0)"]\ar[r,shift right=1ex,"F(f_k\pi_1)"',"\vleq"] & FY_k
\end{tikzcd}$ and then the $n$-ary pullback
\begin{displaymath}
\begin{tikzcd}
	& & Q\ar[dll,tail,"\mu_1"']\ar[dl,tail,"\mu_2"]\ar[drr,tail,"\mu_n"] & & \\
	I_{1}'\ar[drr,tail,"{i_{1}'}"'] & I_{2}'\ar[dr,tail,"{i_{2}'}"] & & \hdots & I_{n}'\ar[dll,tail,"{i_{n}'}"] \\
	& & FP & &
\end{tikzcd}
\end{displaymath}
so that $i'=i_{k}'\mu_k$ for all $k=1,...,n$.

Now we know that for all $k=1,...,n$ the unique $q_k\colon FE_k\to I_{k}'$ such that $i_{k}'q_k=Fe_k$ is an effective epi. Hence, for all $k$ we have a commutative square
\begin{displaymath}
\begin{tikzcd}
	FE_k\ar[r,"Fe_k"]\ar[d,two heads,"q_k"'] & FP\ar[d,equal] \\
	I_{k}'\ar[r,tail,"{i_{k}'}"'] & FP
\end{tikzcd}
\end{displaymath}
It follows from \ref{useful lemma} that the induced $q\colon\Lambda'\to Q$ with $\mu_kq=q_k\lambda_{k}'$ is an effective epi.

Finally, it can now be seen that the comparison morphism between the image of weak joint kernel and the strong joint kernel is the composition $\begin{tikzcd}
	FE\ar[r,two heads,"r"] & \Lambda'\ar[r,two heads,"q"] & Q\ar[r,two heads,"{p'}"] & \Lambda
\end{tikzcd}$
and hence is an effective epi.
\end{proof}

Finally, we will need to know that in the lex context left covering functors reduce to those preserving finite limits.

\begin{prop}
	Let $F\colon\ca{C}\to\ca{E}$ be functor from the lex category $\ca{C}$ to the regular category $\ca{E}$. If $F$ is left covering, then it is left exact.
\end{prop}
\begin{proof}
The main point here is to observe that a left covering functor always preserves jointly order-monomorphic finite families of arrows. To see this, consider arrows $f_k\colon X\to Y_k$ in $\ca{C}$, for $k=1,...,n$. We can then form the (weak) \emph{joint kernel} of this family.  
To say that the family $(f_k)_{k=1}^{n}$ is jointly order-monomorphic is equivalent to having $l_0\leq l_1$. If $\begin{tikzcd}
	C\ar[r,shift left=0.75ex,"c_0"]\ar[r,shift right=0.75ex,"c_1"'] & FX
\end{tikzcd}$ is the corresponding joint kernel of $(Ff_k)_{k=1}^{n}$ in $\ca{E}$, then we have that the comparison $q\colon FL\to C$ is an effective epimorphism, since $F$ is left covering with respect to joint kernels. Hence, $q$ is in particular an order-epimorphism. Thus, we now have that $l_0\leq l_1\implies Fl_0\leq Fl_1\implies c_0q\leq c_1q \implies c_0\leq c_1$, and so $(Ff_k)_{k=1}^{n}$ are also jointly order-monomorphic.

\end{proof}

We would like to show that every left covering functor $F\colon\ca{C}\to\ca{E}$ into an exact category $\ca{E}$ extends uniquely (up to isomorphism) to a regular functor $\overline{F}\colon\ca{C}_{\rm{ex}}\to\ca{E}$. The definition (and hence the essential uniqueness) of $\overline{F}$ is forced upon us by the desired properties of $\overline{F}$.

Indeed, from \ref{C generates C_ex via coinserters} we know that every $(X,R)\in\ca{C}_{\rm{ex}}$ occurs in a coinserter diagram of the form \begin{tikzcd}[cramped]
	\Gamma R\ar[r,shift left=0.75ex,"\Gamma r_0"]\ar[r,shift right=0.75ex,"\Gamma r_1"'] & \Gamma X\ar[r,two heads,"{[1_X]}"] & (X,R)\end{tikzcd}. 
Applying $\overline{F}$ to this diagram and using the isomorphism $\overline{F}\circ\Gamma\cong F$, we obtain the following diagram in $\ca{E}$
\begin{center}
	\begin{tikzcd}
		FR\ar[rr,shift left=0.75ex," Fr_0"]\ar[rr,shift right=0.75ex,"Fr_1"']\ar[dr,two heads,"p"'] & & FX\ar[r,two heads,"{\overline{F}[1_X]}"] & \overline{F}(X,R) \\
		& R'\ar[ur,shift left=0.75ex,"{r_{0}'}"]\ar[ur,shift right=0.75ex,"{r_{1}'}"'] & &
	\end{tikzcd}
\end{center}
where $(r_{0}',r_{1}')$ is the image of $(Fr_0,Fr_1)$ in $\ca{E}$, which is a congruence by \ref{left covering functors and congruences}. By regularity of $\overline{F}$ we have that $\overline{F}[1_X]$ is an $\nc{so}$-morphism and that $R'$ is its kernel congruence. In particular, $\overline{F}(X,R)$ must be defined as the coinserter of $(Fr_0,Fr_1)$. Let us denote this coinserter $FX\twoheadrightarrow\overline{F}(X,R)$ by $q_{(X,R)}$. 

Similarly, if $[f]\colon(X,R)\to(Y,S)$ is a morphism in $\ca{C}_{\rm{ex}}$, then $\overline{F}[f]$ must be defined as the unique morphism making the following diagram commute
\begin{center}
	\begin{tikzcd}
		FR\ar[r,shift left=0.75ex," Fr_0"]\ar[r,shift right=0.75ex,"Fr_1"']\ar[d,"F\bar{F}"'] & FX\ar[r,two heads,"q_{(X,R)}"]\ar[d,"Ff"] & \overline{F}(X,R)\ar[d,dashed,"{\overline{F}[f]}"] \\
		FS\ar[r,shift left=0.75ex," Fs_0"]\ar[r,shift right=0.75ex,"Fs_1"'] & FY\ar[r,two heads,"q_{(Y,S)}"'] & \overline{F}(Y,S)
	\end{tikzcd}
\end{center}
Note that $\overline{F}[f]$ exists because both rows are coinserters in $\ca{E}$.

This deals with the existence and essential uniqueness of the functor $\overline{F}\colon\ca{C}_{\rm_{ex}}\to\ca{E}$. The crux of the universal property, which is the proof that $\overline{F}$ so defined is indeed a regular functor, is officially contained in the following.

\begin{thm}
	For any left covering functor $F\colon\ca{C}\to\ca{E}$ into an exact category $\ca{E}$, there is a unique up to isomorphism regular functor $\overline{F}\colon\ca{C}_{\rm{ex}}\to\ca{E}$ such that $\overline{F}\circ\Gamma\cong F$.
\end{thm}
\begin{proof}
	To show that $\overline{F}$ is left exact, it suffices to show that it is left covering with respect to finite products and inserters.
	\vspace{0.3cm}
	
	\underline{Binary Products}: Let $(X,R)$ and $(Y,S)$ be objects of $\ca{C}_{\rm{ex}}$. Recall that the corresponding binary product $\begin{tikzcd}[cramped] (X,R) & (P,\Gamma)\ar[l,"{[\pi_X]}"']\ar[r,"{[\pi_Y]}"] & (Y,S)
	\end{tikzcd}$ is constructed by first taking a weak binary product
	$\begin{tikzcd}[cramped] X & P\ar[l,"\pi_X"']\ar[r,"\pi_Y"] & Y
	\end{tikzcd}$ in $\ca{C}$. Then the following diagram commutes, where $\phi$ and $\psi$ are the canonical comparison morphisms, as can be seen by post-composing with the two projections from the product.
	\begin{center}
		\begin{tikzcd}
			FP\ar[r,two heads,"\phi"]\ar[d,"q_{(P,\Gamma)}"'] & FX\times FY\ar[d,two heads,"q_{(X,R)}\times q_{(Y,S)}"] \\
			\overline{F}(P,\Gamma)\ar[r,"\psi"'] & \overline{F}(X,R)\times\overline{F}(Y,S)
		\end{tikzcd}
	\end{center} 
	
	Observe that $q_{(X,R)}\times q_{(Y,S)}$ is an $\nc{so}$-morphism because both $q_{(X,R)}$ and $q_{(Y,S)}$ are such and we are in a regular category, while $\phi$ is an $\nc{so}$-morphism because $F$ is left covering. Thus, their composition is also an $\nc{so}$-morphism and then by the commutativity above so is $\psi$.
	\vspace{0.3cm}
	
	\underline{Terminal Object}: The terminal object of $\ca{C}_{\rm{ex}}$ is $(T,T\times T)$, where $T$ is any weak terminal object of $\ca{C}$ and $T\times T$ is any weak binary product in $\ca{C}$. The unique arrow $FT\twoheadrightarrow 1$ in $\ca{E}$ is an $\nc{so}$-morphism, hence so is $\overline{F}(T,T\times T)\to 1$.
	\vspace{0.3cm}
	
	\underline{Inserters}: Consider an inserter $\begin{tikzcd}[cramped] (E,\tilde{R})\ar[r,tail,"{[e]}"] & (X,R)\ar[r,shift left=1ex,"{[f]}"]\ar[r,shift right=1ex,"{[g]}"']\ar[r,phantom,"\vleq"] & (Y,S)
	\end{tikzcd}$ in $\ca{C}_{\rm{ex}}$. Recall that this is constructed by first taking the weak conical limit
	\begin{center}
		\begin{tikzcd}
			& X\ar[dl,"f"']\ar[dr,"g"] & \\
			Y & E\ar[u,dashed,"e"']\ar[d,dashed,"\phi"] & Y \\
			& S\ar[ul,"s_0"]\ar[ur,"s_1"'] & 
		\end{tikzcd}
	\end{center} 
	in $\ca{C}$. We then also consider in $\ca{E}$ the following corresponding inserters
	\begin{center}
		\begin{tikzcd} I\ar[r,tail,"i"] & \overline{F}(X,R)\ar[r,shift left=1ex,"{\overline{F}[f]}"]\ar[r,shift right=1ex,"{\overline{F}[g]}"']\ar[r,phantom,"\vleq"] & \overline{F}(Y,S)
		\end{tikzcd}
	\end{center}
	\begin{center}
		\begin{tikzcd} I'\ar[r,tail,"{i'}"] & FX\ar[r,shift left=1ex,"q_{(X,R)}Ff"]\ar[r,shift right=1ex,"q_{(Y,S)}Fg"']\ar[r,phantom,"\vleq"] & \overline{F}(Y,S)
		\end{tikzcd}
	\end{center}
	We observe here that, by exactness of $\ca{E}$, the inserter $I'$ can also be constructed as the first of the following two successive pullbacks.
	\begin{center}
		\begin{tikzcd}
			I''\ar[r,two heads,"{p'}"]\ar[d,"\psi"'] & I'\ar[r,tail,"{i'}"]\ar[d,"{\psi'}"'] & FX\ar[d,"{\langle Ff,Fg\rangle}"] \\
			FS\ar[r,two heads,"p"'] & S'\ar[r,tail,"{\langle s_0,s_1\rangle}"'] & FY\times FY
		\end{tikzcd}
	\end{center}
	In addition, notice that $I''$ is now precisely the following limit in $\ca{E}$:
	\begin{center}
		\begin{tikzcd}
			& FX\ar[dl,"Ff"']\ar[dr,"Fg"] & \\
			FY & I''\ar[u,dashed,"{i'p'}"']\ar[d,dashed,"\psi"] & FY \\
			& FS\ar[ul,"Fs_0"]\ar[ur,"Fs_1"'] & 
		\end{tikzcd}
	\end{center}
	Thus, since $F$ is left covering, the morphism $\epsilon\colon FE\twoheadrightarrow I''$ with $i'p'\epsilon=Fe$ and $\psi\epsilon=F\phi$ is an $\nc{so}$-morphism in $\ca{E}$. Now we gather all this data in the following diagram:
	\begin{center}
		\begin{tikzcd}
			F\tilde{R}\ar[rr,shift left=0.75ex,"F\tilde{r}_0"]\ar[rr,shift left=0.75ex,"F\tilde{r}_1"']\ar[dd,"F\bar{e}"'] & & FE\ar[rr,two heads,"q_{(X,E)}"]\ar[dd,"Fe"']\ar[dr,two heads,"{p'\epsilon}"]  & & \overline{F}(E,\tilde{R})\ar[dd,near start,"{\overline{F}[e]}"]\ar[dr,"\bar{\epsilon}"] &  \\
			& & & I'\ar[dl,tail,"{i'}"]\ar[rr,dashed,near start,"r"] & & I\ar[dl,tail,"i"] \\
			FR\ar[rr,shift left=0.75ex,"Fr_0"]\ar[rr,shift right=0.75ex,"Fr_1"']\ar[d,shift left=0.75ex,"F\bar{f}"]\ar[d,shift right=0.75ex,"F\bar{g}"'] & & FX\ar[rr,two heads,"q_{(X,R)}"]\ar[d,shift left=0.75ex,"Ff"]\ar[d,shift right=0.75ex,"Fg"'] & & \overline{F}(X,R)\ar[d,shift left=0.75ex,"{\overline{F}[f]}"]\ar[d,shift right=0.75ex,"{\overline{F}[g]}"'] & \\
			FS\ar[rr,shift left=0.75ex,"Fs_0"]\ar[rr,shift right=0.75ex,"Fs_1"']\ar[dr,two heads,"p"'] & & FY\ar[rr,two heads,"q_{(Y,S)}"'] & & \overline{F}(Y,S) & \\
			& S'\ar[ur,shift left=0.75ex,"{s_{0}'}"]\ar[ur,shift right=0.75ex,"{s_{1}'}"'] & & & & 
		\end{tikzcd}
	\end{center}
	Observe that the top right parallelogram commutes because it does so after post-composing with the (order-)mono $i$. Hence, since $p'\epsilon$ is an $\nc{so}$-morphism, for $\bar{e}$ to be an $\nc{so}$-morphism as desired, it suffices to show that the induced $r\colon I'\to I$ is such. But finally this follows from the following diagram where the rows are inserters and an application of \ref{useful lemma}.
	\begin{center}
		\begin{tikzcd}
			I'\ar[r,tail,"{i'}"]\ar[d,"r"'] & FX\ar[r,shift left=1ex,"q_{(Y,S)}Ff"]\ar[r,shift right=1ex,"q_{(Y,S)}Fg"']\ar[r,phantom,"\vleq"]\ar[d,two heads,"q_{(X,R)}"'] & \overline{F}(Y,S)\ar[d,equal] \\
			I\ar[r,tail,"i"'] & \overline{F}(X,R)\ar[r,shift left=1ex,"{\overline{F}[f]}"]\ar[r,shift right=1ex,"{\overline{F}[g]}"']\ar[r,phantom,"\vleq"] & \overline{F}(Y,S)
		\end{tikzcd}
	\end{center}
\end{proof}

We now collect some basic observations on properties of the extension functor $\overline{F}\colon\ca{C}_{\rm{ex}}\to\ca{E}$. In particular, this characterizes when $\overline{F}$ is an equivalence. Below, by an \emph{order-faithful} functor $F\colon\ca{C}\to\ca{E}$ we mean one which locally reflects the order of morphisms, i.e. $Ff\leq Fg\implies f\leq g$ for all $f,g\colon X\to Y\in\ca{C}$.

\begin{prop}\label{functoriality properties of C_ex}
	Let $F\colon\ca{C}\to\ca{E}$ be a left covering functor  into an exact category $\ca{E}$, so that $\overline{F}\colon\ca{C}_{\rm{ex}}\to\ca{E}$ is defined.
	\begin{enumerate}
		\item If $F$ is fully order-faithful and $FX$ is projective for all $X\in\ca{C}$, then $\overline{F}$ is fully order-faithful.
		\item If, in addition to the conditions in (1), for every $A\in\ca{E}$ there exists an $\nc{so}$-morphism $FX\twoheadrightarrow A$, then $\overline{F}$ is an equivalence.
	\end{enumerate}
\end{prop}
\begin{proof}
(1) Consider morphisms $\begin{tikzcd}[cramped](X,R)\ar[r,shift left=0.75ex,"{[f]}"]\ar[r,shift right=0.75ex,"{[g]}"'] & (Y,S) \end{tikzcd}$ in $\ca{C}_{\rm{ex}}$ such that $\overline{F}[f]\leq\overline{F}[g]$ in $\ca{E}$. Then we have $q_{(Y,S)}\circ\overline{F}[f]\leq q_{(Y,S)}\circ\overline{F}[g]$, so there is a $u\colon FX\to\tilde{S}\in\ca{E}$ such that $\tilde{s}_0u=Ff$ and $\tilde{s}_1u=Fg$, since in the exact category $\ca{E}$, $\tilde{S}$ is the kernel congruence of its coinserter $q_{(Y,S)}$.
\begin{center}
	\begin{tikzcd}
		FR\ar[r,two heads,"p_R"]\ar[d,shift left=0.75ex,"F\bar{g}"]\ar[d,shift right=0.75ex,"F\bar{f}"'] & \tilde{R}\ar[d,shift left=0.75ex,"\tilde{g}"]\ar[d,shift right=0.75ex,"\tilde{f}"']\ar[r,shift left=0.75ex,"\tilde{r}_0"]\ar[r,shift right=0.75ex,"\tilde{r}_1"'] & FX\ar[r,two heads,"q_{(X,R)}"]\ar[d,shift left=0.75ex,"Fg"]\ar[d,shift right=0.75ex,"Ff"']\ar[dl,dashed]\ar[dll,dashed] & \overline{F}(X,R)\ar[d,shift left=0.75ex,"{\overline{F}[g]}"]\ar[d,shift right=0.75ex,"{\overline{F}[f]}"'] \\
		FS\ar[r,two heads,"p_S"'] & \tilde{S}\ar[r,shift left=0.75ex,"\tilde{s}_0"]\ar[r,shift right=0.75ex,"\tilde{s}_1"'] & FY\ar[r,two heads,"q_{(Y,S)}"'] & \overline{F}(Y,S)
	\end{tikzcd}
\end{center}
Now because $FX$ is $\nc{so}$-projective and $p_S$ is an $\nc{so}$-morphism, we obtain a $v\colon FX\to FS$ such that $p_{S}v=u$. Then $F$ being full implies the existence of a $\Sigma\colon X\to S$ with $F\Sigma=v$. In turn, we now have $F(s_{0}\Sigma)=Fs_{0}F\Sigma=\tilde{s}_{0}p_{S}v=\tilde{s}_{0}u=Ff$ and $F(s_{1}\Sigma)=Fs_{1}F\Sigma=\tilde{s}_{1}p_{S}v=\tilde{s}_{1}u=Fg$, hence by faithfulness of $F$ we obtain $s_{0}\Sigma=f$ and $s_{1}\Sigma=g$, so that $[f]\leq[g]$.

Now consider any morphism $h\colon\overline{F}(X,R)\to\overline{F}(Y,S)$ in $\ca{E}$. Since $FX$ is projective, there exists a $u\colon FX\to FY\in\ca{E}$ such that $q_{(Y,S)}u=hq_{(X,R)}$. Then we have $q_{(Y,S)}u\tilde{r}_{0}=hq_{(X,R)}\tilde{r}_0\leq hq_{(X,R)}\tilde{r}_1=q_{(Y,S)}u\tilde{r}_{1}$, which implies the existence of a unique $\tilde{u}\colon\tilde{R}\to\tilde{S}$ such that $\tilde{s}_0\tilde{u}=u\tilde{r}_0$ and $\tilde{s}_1\tilde{u}=u\tilde{r}_1$. Since $FR$ is projective, we obtain a $\overline{u}\colon FR\to FS$ with $p_{S}\overline{u}=\tilde{u}p_{R}$.
\begin{center}
\begin{tikzcd}
	FR\ar[r,two heads,"p_R"]\ar[d,dashed,"\overline{u}"'] & \tilde{R}\ar[d,dashed,"\tilde{u}"']\ar[r,shift left=0.75ex,"\tilde{r}_0"]\ar[r,shift right=0.75ex,"\tilde{r}_1"'] & FX\ar[r,two heads,"q_{(X,R)}"]\ar[d,dashed,"u"] & \overline{F}(X,R)\ar[d,"h"] \\
	FS\ar[r,two heads,"p_S"'] & \tilde{S}\ar[r,shift left=0.75ex,"\tilde{s}_0"]\ar[r,shift right=0.75ex,"\tilde{s}_1"'] & FY\ar[r,two heads,"q_{(Y,S)}"'] & \overline{F}(Y,S)
\end{tikzcd}
\end{center}

By fullness of $F$, there exist morphisms $f\colon X\to Y$ and $\overline{f}\colon R\to S$ in $\ca{C}$ such that $Ff=u$ and $F\overline{f}=\overline{u}$. Now we have that
\begin{displaymath}
	F(fr_0)=FfFr_0=u\tilde{r}_0p_{R}=\tilde{s}_0p_{S}\overline{u}=Fs_0F\overline{f}=F(s_0\overline{f})
\end{displaymath}
which by faithfulness of $F$ implies $fr_0=s_0\overline{f}$, and similarly that $fr_1=s_1\overline{f}$. That is to say that we have a morphism $[f]\colon(X,R)\to(Y,S)$ in $\ca{C}_{\rm{ex}}$, which then by construction satisfies $\overline{F}([f])=h$.
\vspace{0.3cm}

(2) It suffices to show that $\overline{F}$ is essentially surjective on objects, so consider any $A\in\ca{E}$. Take an $\nc{so}$-morphism $e\colon FX\twoheadrightarrow A$. Then consider the kernel congruence $\begin{tikzcd}
K\ar[r,shift left=0.75ex,"k_0"]\ar[r,shift right=0.75ex,"k_1"'] & FX 
\end{tikzcd}$ of $e$ and then again cover with an $\nc{so}$-morphism $q\colon FR\twoheadrightarrow K$. By fullness of $F$, there are morphisms $\begin{tikzcd}
R\ar[r,shift left=0.75ex,"r_0"]\ar[r,shift right=0.75ex,"r_1"'] & X 
\end{tikzcd}$ such that $Fr_0=k_0q$ and $Fr_1=k_1q$. We claim that this pair forms a pseudocongruence on $X$ in $\ca{C}$.

To that end, consider the following two diagrams, the first being a weak comma in $\ca{C}$ and the second a comma in $\ca{E}$.
\begin{center}
\begin{tikzcd}
	I_X\ar[r,"c_1"]\ar[d,"c_0"']\ar[dr,phantom,"\leq"] & X\ar[d,equal] \\
	X\ar[r,equal] & X
\end{tikzcd}
\qquad
\begin{tikzcd}
	I_{FX}\ar[r,"i_1"]\ar[d,"i_0"']\ar[dr,phantom,"\leq"] & FX\ar[d,equal] \\
	FX\ar[r,equal] & FX
\end{tikzcd}
\end{center}
Let $p\colon FI_X\twoheadrightarrow I_{FX}$ be the canonical comparison morphism. Since  $\begin{tikzcd}
	K\ar[r,shift left=0.75ex,"k_0"]\ar[r,shift right=0.75ex,"k_1"'] & FX 
\end{tikzcd}$ is order-reflexive, there exists a $u\colon I_{FX}\to K$ such that $k_0u=i_0$ and $k_1u=i_1$. By $FX$ being projective, there is now a $v\colon FI_X\to FR$ with $qv=up$. In turn, by $F$ being full we obtain an $n\colon I_X\to R\in\ca{C}$ such that $Fn=v$. Now we have $F(r_0n)=Fr_0 Fn=k_0qv=k_0up=i_0p=Fc_0$ and so by faithfulness of $F$ we deduce that $r_0n=c_0$, and then similarly that $r_1n=c_1$. Hence, $R$ is order-reflexive.

For transitivity, consider the following squares, where the first is a weak pullback in $\ca{C}$ and the other two are pullbacks in $\ca{E}$. 
\begin{center}
\begin{tikzcd}
	R*R\ar[r,"d_{0}^{R}"]\ar[d,"d_{1}^{R}"'] & R\ar[d,"r_1"] \\
	R\ar[r,"r_0"'] & X
\end{tikzcd}
\qquad
\begin{tikzcd}
	K*K\ar[r,"d_{0}^{K}"]\ar[d,"d_{1}^{K}"'] & K\ar[d,"k_1"] \\
	K\ar[r,"k_0"'] & FX
\end{tikzcd}
\qquad
\begin{tikzcd}
	FR*FR\ar[r,"d_{0}^{FR}"]\ar[d,"d_{1}^{FR}"'] & FR\ar[d,"Fr_1"] \\
	FR\ar[r,"Fr_0"'] & FX
\end{tikzcd}
\end{center}

Denote by $\sigma\colon FR*FR\to K*K$ the induced morphism in $\ca{E}$, i.e. the one uniquely determined by the equalities $d_{1}^{K}\sigma=qd_{1}^{FR}$ and $d_{0}^{K}\sigma=qd_{0}^{FR}$. Let $p\colon F(R*R)\to FR*FR$ be the canonical comparison morphism. Since $F(R*R)$ is projective in $\ca{E}$, there exists a $w\colon F(R*R)\to FR$ with $qw=\tau^{K}\sigma p$, where $\tau^{K}$ is the morphism exhibiting transitivity of the congruence $K$. By fullness of $F$, there is then a $\tau^{R}\colon R*R\to R\in\ca{C}$ with $F\tau^{R}=w$. Now we have that
\begin{displaymath}
	F(r_0\tau^{R})=Fr_0F\tau^{R}=k_0qw=k_0\tau^{K}\sigma p=k_0d_{0}^{K}\sigma p=k_0qd_{0}^{FR}p=Fr_0Fd_{0}^{R}=F(r_0d_{0}^{R})
\end{displaymath}
so by faithfulness of $F$ we deduce $r_0\tau^{R}=r_0d_{0}^{R}$. Similarly, we also obtain $r_1\tau^{R}=r_1d_{1}^{R}$, so that $R$ is transitive.

Finally, it is now immediate from the definition of $\overline{F}$ and the construction of $R$ that we have $\overline{F}(X,R)\cong A$
\end{proof}

Since the inclusion $\ca{P}\hookrightarrow\ca{E}$ of a projective cover $\ca{P}$ into an exact category $\ca{E}$ satisfies all the properties in the above proposition, we immediately have the following.

\begin{thm}\label{ExactCompOfProjCover}
	Let $\ca{E}$ be an exact category and suppose that $\ca{P}$ is a projective cover of $\ca{E}$. Then $\ca{P}_{\rm{ex}}\simeq\ca{E}$.
\end{thm}

Thus, the characterization of the exact completion and its relation to projectives has a precise analogue in the $\nc{Pos}$-enriched world, as long as all notions are interpreted in the appropriate enriched sense. In particular, if two exact categories have equivalent projective covers, then they are themselves equivalent.

\begin{cor}
	Let $\ca{E}$ and $\ca{F}$ be exact categories with projective covers $\ca{P}$ and $\ca{Q}$ respectively. If $\ca{P}\simeq\ca{Q}$, then $\ca{E}\simeq\ca{F}$.
\end{cor}

\begin{ex}
	We shall have a large class of examples of categories which are exact completions by virtue of the results in the next section. However, let us also mention a couple here.
\begin{enumerate}
	\item Every set, considered as a discrete poset, is projective in $\nc{Pos}$. And in turn every poset $(X,\leq_X)$ can be covered with an obvious surjection from a discrete poset, namely $(X,=)\twoheadrightarrow(X,\leq_X)$. Thus, $\nc{Set}$ is a projective cover in $\nc{Pos}$, in fact the full subcategory of projectives, so we deduce that $\nc{Set}_{\rm{ex}}\simeq\nc{Pos}$. We note here that the exact completion of $\nc{Set}$ qua ordinary category is itself.
	
	\item Recall that a compact Hausdorff space $X$ is said to be \emph{extremally disconnected} if the closure in $X$ of every open subset remains open. By the work of Gleason \cite{GleasonProjective}, we know that these are precisely the projective objects in the ordinary category $\nc{CHaus}$ of compact Hausdorff spaces and that the latter category has enough projectives. Then it is clear that, if we consider the category $\nc{ExtrDis}$ of extremally disconnected compact Hausdorff spaces as a locally discrete subcategory of $\nc{Nach}$, i.e. by equipping every space $X\in\nc{ExtrDis}$ with the equality order, then it is a projective cover therein. Thus, we deduce that $\nc{ExtrDis}_{\rm{ex}}\simeq\nc{Nach}$. 
\end{enumerate}
\end{ex}

\section{Exact Completion and Categories of Ordered Algebras}

In this section we connect the exact completion with varieties of ordered algebras and thus answer an open question of Kurz \& Velebil from \cite{Kurz-Velebil}. More generally, we will apply the exact completion to obtain a characterization of categories which are monadic (in the enriched sense) over $\nc{Pos}$, following Vitale's approach in \cite{VitaleMonadicOverSet} for the ordinary setting. However, it is clear that in our present context we must restrict to sufficiently ``nice'' monads in order to reconcile monadicity with exactness. Indeed, as observed for example in \cite{Kurz-Velebil}, a category of the form $\nc{Pos}^{\mathbb{T}}$ may be badly behaved. In particular, such a category need not even be regular. The reason that such issues do not occur in the ordinary setting is essentially because of the Axiom of Choice in $\nc{Set}$.

The monads we will need to restrict to are those which preserve quotients of congruences. From one perspective, this is on a practical level the minimum requirement that one needs to impose so that the proof of exactness of $\nc{Pos}^{\mathbb{T}}$ works. From another point of view, it is known \cite{Kurz-Velebil,Adamek_Dostal_Velebil_2022} that ordered varieties correspond precisely to those monads which are \emph{strongly finitary}. And in turn, the latter property is equivalent to being finitary and preserving quotients of congruences. So it seems reasonable that dropping the finitary part should correspond to general (exact) categories $\nc{Pos}^{\mathbb{T}}$.

We begin with an easy observation on projectives in a category of algebras $\nc{Pos}^{\mathbb{T}}$, at least for sufficiently nice monads $\mathbb{T}$.

\begin{lem}\label{ProjectiveAlg}
	Let $\mathbb{T}=(T,\eta,\mu)$ be a monad on $\nc{Pos}$ such that $T$ preserves $\nc{so}$-morphisms (=surjections). Then:
	\begin{enumerate}
		\item For every discrete poset $X$, the free algebra $(TX,\mu_X)$ is an $\nc{so}$-projective in $\nc{Pos}^{\mathbb{T}}$.
		\item $\nc{Pos}^{\mathbb{T}}$ has enough projectives.
	\end{enumerate}
\end{lem}
\begin{proof}
	(1) Consider any morphism $f\colon (TX,\mu_X)\to(B,\beta)$ and an $\nc{so}$-morphism $q\colon(A,\alpha)\twoheadrightarrow(B,\beta)$. Since $T$ preserves $\nc{so}$-morphisms, so does the forgetful functor $U\colon\nc{Pos}^{\mathbb{T}}\to\nc{Pos}$ by \cite[Proposition 4.12]{Kurz-Velebil}, i.e. $q$ is surjective. Thus, since $X$ is discrete, any function $\bar{f}\colon X\to A$ with $q\bar{f}=f\eta_X$ is automatically order-preserving and so lives in $\nc{Pos}$. Now the induced morphism of $\mathbb{T}$-algebras $\alpha\circ T\bar{f}\colon(TX,\mu_X)\to(A,\alpha)$ satisfies
	\begin{displaymath}
		q\alpha T\bar{f}=\beta TqT\bar{f}=\beta T(q\bar{f})=\beta T(f\eta_X)=\beta TfT\eta_X=f\mu_X T\eta_X=f
	\end{displaymath}
	
	(2) For any $(A,\alpha)\in\nc{Pos}^{\mathbb{T}}$ we have the canonical morphism $\alpha\colon(TA,\mu_A)\to(A,\alpha)$ such that $U\alpha$ is a (split) surjection. Since $U\colon\nc{Pos}^{\mathbb{T}}\to\nc{Pos}$ also reflects $\nc{so}$-morphisms, again by \cite[Proposition 4.12]{Kurz-Velebil}, we have that $\alpha$ is an $\nc{so}$-morphism in $\nc{Pos}^{\mathbb{T}}$. Consider also the canonical surjection $p\colon|A|\twoheadrightarrow A$, where $|A|$ denotes the discrete poset with the same underlying set as $A$. Applying the free algebra functor $F\colon\nc{Pos}\to\nc{Pos}^{\mathbb{T}}$, which is a left adjoint, we obtain an $\nc{so}$-morphism $Fp\colon F|A|\to FA$. Then the composition $\begin{tikzcd}
		F|A|\ar[r,two heads,"Fp"] & FA\ar[r,two heads,"\alpha"] & (A,\alpha)
	\end{tikzcd}$ is an $\nc{so}$-morphism as well, while $F|A|=(T|A|,\mu_{|A|})$ is projective by part (1).
\end{proof}

If $\nc{Kl}(\mathbb{T})_{d}$ denotes the full subcategory of the Kleisli category on the algebras which are free on a discrete poset, then the above lemma shows that $\nc{Kl}(\mathbb{T})_{d}$ is a projective cover of $\nc{Pos}^{\mathbb{T}}$. When moreover $\mathbb{T}$ preserves coinserters of congruences, $\nc{Pos}^{\mathbb{T}}$ is exact and so we immediately have the following.

\begin{cor}
	Let $\mathbb{T}$ be a monad on $\nc{Pos}$ which preserves coinserters of congruences. Then $\nc{Pos}^{\mathbb{T}}\simeq(\nc{Kl}(\mathbb{T})_{d})_{\rm{ex}}$
\end{cor}

In particular, since the monads corresponding to varieties of ordered algebras satisfy the condition above, we also obtain the following.

\begin{cor}
	If $\ca{V}$ is a variety of ordered algebras, then $\ca{V}\simeq \ca{P}_{\rm{ex}}$, where $\ca{P}$ is the full subcategory of $\ca{V}$ on the algebras which are free on a discrete poset.
\end{cor}

Before the main result of this section, we will need the following easy characterization of Kleisli categories for monads on $\nc{Pos}$. This again precisely mimics the ordinary case of monads over $\nc{Set}$.

\begin{lem}\label{Kleisli characterization}
	For a category $\ca{C}$, the following are equivalent:
	\begin{enumerate}
		\item $\ca{C}\simeq\nc{Kl}(\mathbb{T})$ for a monad $\mathbb{T}$ on $\nc{Pos}$.
		\item There exists an object $G\in\ca{C}$ such that:
		\begin{itemize}
			\item $P\bullet G$ exists for all $P\in\nc{Pos}$.
			\item $(\forall C\in\ca{C})(\exists P\in\nc{Pos})\hspace{5mm}C\cong P\bullet G$
		\end{itemize}
	\end{enumerate}
\end{lem}
\begin{proof}
	If $\ca{C}\simeq\nc{Kl}(\mathbb{T})$, then we can simply take $G=F1$.
	
	For the converse, the existence of all $P\bullet G$ implies that there is a functor $-\bullet G\colon\nc{Pos}\to\ca{C}$, which is then by definition left adjoint to $\ca{C}(G,-)$ and hence there is an associated monad $\mathbb{T}$. The comparison functor $\nc{Kl}(\mathbb{T})\to\ca{C}$ is always fully order-faithful and then the second condition says that it is also essentially surjective on objects, hence an equivalence of $\nc{Pos}$-categories.
\end{proof}

We are now ready for the main application.

\begin{thm}\label{characterization of monadic over Pos}
	For a category $\ca{E}$, the following are equivalent:
	\begin{enumerate}
		\item $\ca{E}\simeq\nc{Pos}^{\mathbb{T}}$ for a monad $\mathbb{T}$ on $\nc{Pos}$ which preserves coinserters of congruences.
		\item $\ca{E}$ is exact and there exists an object $G\in\ca{E}$ such that:
		\begin{itemize}
			\item $P\bullet G$ exists for all $P\in\nc{Pos}$.
			\item $G$ is an $\nc{so}$-projective.
			\item For every $X\in\ca{E}$, there exists an $\nc{so}$-morphism $P\bullet G\twoheadrightarrow X$ for some $P\in\nc{Pos}$.
		\end{itemize}
	\end{enumerate}
\end{thm}
\begin{proof}

The implication (1)$\implies$(2) is more or less contained in \cite{Kurz-Velebil}, so we only prove the converse.

Let $\ca{C}$ be the full subcategory of $\ca{E}$ spanned by all objects of the form $P\bullet G$, for any $P\in\nc{Pos}$. Then we have $\ca{C}\simeq\nc{Kl}(\mathbb{T})$ for a monad $\mathbb{T}$ on $\nc{Pos}$, namely the one induced by the adjunction $-\bullet G\dashv\ca{E}(G,-)$. We will show that this monad preserves coinserters of congruences, for which it suffices of course to show that $\ca{E}(G,-)$ preserves them. 

So consider a coinserter $\begin{tikzcd}
	R\ar[r,shift left=0.75ex,"r_0"]\ar[r,shift right=0.75ex,"r_1"'] & X\ar[r,two heads,"q"] & Y
\end{tikzcd}$ in $\ca{E}$, where $(r_0,r_1)$ is a congruence. Since $\ca{E}$ is exact, $(r_0,r_1)$ must in fact be the kernel congruence of its coinserter $q$. Since $\ca{E}(G,-)$ preserves limits, in the diagram $\begin{tikzcd}
\ca{E}(G,R)\ar[r,shift left=0.75ex,"r_0\circ-"]\ar[r,shift right=0.75ex,"r_1\circ-"'] & \ca{E}(G,X)\ar[r,two heads,"q\circ-"] & \ca{E}(G,Y)
\end{tikzcd}$ in $\nc{Pos}$, $(r_0\circ-,r_1\circ-)$ is the kernel congruence of $q\circ-$. But $G$ is projective, so $\ca{E}(G,-)$ also preserves $\nc{so}$-morphisms, whence $q\circ-$ is also an $\nc{so}$-morphism in $\nc{Pos}$. Thus, $q\circ-$ must be the coinserter of its kernel congruence.

Now let $\ca{C}_d$ be the full subcategory of $\ca{C}$ on the copowers $I\bullet G$ with $I$ a discrete poset, i.e. on the coproducts $\bigsqcup_{I}G$. The equivalence $\ca{C}\simeq\nc{Kl}(\mathbb{T})$ clearly restricts to an equivalence $\ca{C}_{d}\simeq\nc{Kl}(\mathbb{T})_{d}$. Then $G$ being projective implies that every object of $\ca{C}_{d}$ is as well. Now $\ca{C}_{d}$ is a projective cover of $\ca{E}$: indeed, for any $P\in\nc{Pos}$ there is a canonical $\nc{so}$-morphism $|P|\twoheadrightarrow P$ with $|P|$ discrete. Thus, for any object $X\in\ca{E}$ we can find a composite $\nc{so}$-morphism $|P|\bullet G\twoheadrightarrow P\bullet G\twoheadrightarrow X$ (since $-\bullet G$ is left adjoint). In addition, $\nc{Kl}(\mathbb{T})_d$ is a projective cover of $\nc{Pos}^{\mathbb{T}}$ by \ref{ProjectiveAlg}. Thus, by \ref{ExactCompOfProjCover} we conclude that $\ca{C}_{d}\simeq\nc{Kl}(\mathbb{T})_{d}\implies(\ca{C}_{d})_{ex}\simeq (\nc{Kl}(\mathbb{T})_{d})_{ex}\implies \ca{E}\simeq\nc{Pos}^{\mathbb{T}}$. 
\end{proof}

\begin{rmk}
	From \ref{characterization of monadic over Pos} we can, for example, recover the fact that the category $\nc{Nach}$ of Nachbin's compact ordered spaces is monadic over $\nc{Pos}$, which is originally due to Flagg \cite{Flagg1997algebraic}. Indeed, the one element space $G= \mathfrak{1}$ is easily seen to satisfy the conditions of the theorem, while we know that $\nc{Nach}$ is exact from \cite[Proposition 5.8]{MyPosExReg}.
\end{rmk}

\begin{rmk}
	As is clear from the proof of \ref{characterization of monadic over Pos}, in the third condition of (2) we can in fact assume that every $X\in\ca{E}$ is covered by a copower $P\bullet G$ with $P$ discrete, i.e. by a coproduct of copies of $G$. Such an object is then furthermore projective, while an arbitrary $P\bullet G$ of course need not be.
\end{rmk}

\bibliographystyle{alpha}
\bibliography{myReferences}
	
\end{document}